\documentclass{article}
\usepackage{amsmath,amsfonts,amssymb,bm,cite,graphicx,algorithm,algorithmic,amsthm,hyperref}
\usepackage{epsfig,subfigure}
\usepackage{fullpage,color}
\newtheorem{theorem}{Theorem}
\newtheorem{definition}{Definition}

\newtheorem{lemma}{Lemma}
\def\eop{{\hfill\vbox{\hrule height .3pt
      \hbox{\vrule width.3pt height 7pt
      \kern 7pt
      \vrule width .3pt}
      \hrule height .3pt}} \par\bigskip}
  
  \usepackage{mathtools}
  \mathtoolsset{showonlyrefs}
  
\DeclareMathOperator{\rank}{rank}

\def\norm #1{\left\|#1\right\|}

\def\twon #1{\left\|#1\right\|_2}

\def\frobn #1{\left\|#1\right\|_{\text{F}}}

\def\abs #1{\left|#1\right|}
\def\inp #1{\left\langle#1\right\rangle}

\def\st{\text{subject to }}

\def\bC{\mathbb{C}}

\def\bR{\mathbb{R}}
\def\bE{\mathbb{E}}
\def\bP{\mathbb{P}}

\def\bN{\mathbb{N}}

\def\bZ{\mathbb{Z}}

\def\m #1{\boldsymbol{#1}}

\def\cA{\mathcal{A}}

\def\cC{\mathcal{C}}
\def\cD{\mathcal{D}}

\def\cG{\mathcal{G}}
\def\cH{\mathcal{H}}
\def\cI{\mathcal{I}}

\def\cN{\mathcal{N}}
\def\cO{\mathcal{O}}
\def\cP{\mathcal{P}}

\def\cS{\mathcal{S}}
\def\cT{\mathcal{T}}

\def\ccn{\mathcal{CN}}

\def\bee{\begin{equation}}
	\def\ene{\end{equation}}

\def\beq{\begin{eqnarray}}
	\def\enq{\end{eqnarray}}

\newtheorem{rem}{Remark}

\def\equ #1{\begin{equation}#1\end{equation}}

\def\sbra #1{\left(#1\right)}
\def\mbra #1{\left[#1\right]}
\def\lbra #1{\left\{#1\right\}}
\def\diag #1{\text{diag}#1}

\def\st {\text{ subject to }}

\begin{document}

\title{Lifted Restricted Isometry Property and Fast Algorithm for Spectral Compressed Sensing With Near-Optimal Sample Complexity}
\date{ }

\footnotetext[1]{
	Department of Mathematics, Hong Kong University of Science and Technology,
	Kowloon, Hong Kong SAR, China.
	E-mail: \texttt{xunmeng.wu@xjtu.edu.cn}, \texttt{jfcai@ust.hk}
}

\footnotetext[2]{
	School of Mathematics and Statistics, Xi'an Jiaotong University,
	Xi'an, Shaanxi, China.
	E-mail: \texttt{zbxu@xjtu.edu.cn},
	\texttt{yangzai@xjtu.edu.cn}
}

\author{
	Xunmeng Wu$^{1}$ \and
	Jian-Feng Cai$^{1}$\thanks{Corresponding author.} \and
	Zongben Xu$^{2}$ \and
	Zai Yang$^{2}$
}

\maketitle

\begin{abstract}
This paper studies the spectral compressed sensing problem, which aims to recover an $r$-spectrally-sparse signal of length $N$ from linear compressive measurements. We reformulate the problem as the recovery of a low-rank Hankel matrix. Motivated by this formulation, we propose a Lifted Restricted Isometry Property (LRIP) for the measurement matrix and show that Gaussian measurement matrices satisfy LRIP under certain conditions. Unlike classical RIP, LRIP imposes a rank constraint in the lifted matrix domain while taking measurements in the signal domain through the adjoint Hankel operator. To solve the resulting nonconvex problem efficiently, we introduce a Fast Iterative Hard Thresholding (FIHT) algorithm. Based on LRIP, we show that FIHT converges linearly to the true signal with high probability using $\mathcal{O}(r \ln r \ln^2 N)$ scaled Gaussian measurements, which is near-optimal up to log factors and improves the state-of-the-art order of $r^2$ for nonconvex Hankel methods. Extensive numerical results demonstrate the superiority of FIHT as compared to state-of-the-art convex methods in terms of phase transitions, computational time, and robustness to noise.
\end{abstract}

\noindent \textbf{Keywords:} spectrally sparse signal, restricted isometry property, low-rank Hankel matrix, sample complexity

\section{Introduction}
Spectral compressed sensing refers to the recovery of a spectrally sparse signal from its compressive measurements. A vector $\m{x}_\star \in \bC^N$ is said to be spectrally sparse if it is a superposition of a few complex exponential functions, namely, its entries satisfy
\equ{
	x_{\star,j} = \sum^r_{k=1}s_k e^{\sbra{\mathrm{i} 2\pi f_k-\tau_k} j}, \quad j = 0,1,\ldots,N-1,
	\label{eq:signal}
}
where $f_k\in[0,1)$, $\tau_k \in \bR$, and $s_k\in\bC$ denote the distinct frequency, damping factor, and nonzero complex amplitude of the $k$-th spectral component, respectively, and $\mathrm{i}=\sqrt{-1}$. The sparsity arises from the fact $r \ll N$. Spectrally sparse signals appear in numerous applications, including magnetic resonance imaging \cite{lustig2007sparse}, nuclear magnetic resonance (NMR) spectroscopy \cite{qu2015accelerated}, array signal processing \cite{yang2018sparse,wu2024coffee}, wireless communication \cite{tse2005fundamentals}, and analog-to-digital conversion \cite{tropp2009beyond}.

In this paper, we aim to recover the signal $\m{x}_\star$ from linear compressive measurement
\equ{
	\m{y} = \m{A}\m{x}_\star,
}
where $\m{A}\in \bC^{m\times N}$ with $m<N$ is the measurement matrix. This problem arises in various contexts, including compressed sensing of spectrally sparse bandlimited signals via the random demodulator \cite{tropp2009beyond}, accelerated NMR spectroscopy using linear subsampling \cite{qu2015accelerated}, and radar imaging with continuous time-frequency shifts \cite{heckel2017generalized}.
Note in \eqref{eq:signal} that the number of real unknowns to determine $\m{x}_\star$ is $4r$, which is smaller than the ambient dimension $N$ since $r \ll N$. Hence, it is possible to recover $\m{x}_\star$ from $\m{y}$. The recovery problem can be formulated as
\equ{
	\min_{\m{x} \in \bC^N} \twon{\m{A}\m{x} - \m{y}}^2, \st \m{x} \text{ is $r$-spectrally-sparse}. \label{eq:modelx}
}
Due to the nonlinear and nonconvex property of the spectral sparsity constraint, it is challenging to directly solve \eqref{eq:modelx}.

To address this difficulty, we exploit the Hankel matrix of the signal for recovery. 
Let $\cH \m{x}$ denote an $n_1 \times n_2$ Hankel matrix formed by using $\m{x}\in\bC^N$ with $n_1+n_2-1=N$, i.e.,
\begin{equation} 
	\cH \m{x}
	= \begin{bmatrix} x_{0} & x_{1} & \dots & x_{n_2-1} \\ x_{1} & x_{2} & \dots & x_{n_2} \\ \vdots & \vdots & \ddots & \vdots \\ x_{n_1-1} & x_{n_1} & \dots & x_{N-1} \end{bmatrix}. \label{eq:Hankel}
\end{equation}
Throughout this paper, we assume $n_1\le n_2$ without loss of generality; otherwise, we consider the transpose of $\cH\m{x}$.
We reformulate problem \eqref{eq:modelx} as a fixed-rank Hankel matrix recovery problem:
\equ{
	\min_{\m{x}} \twon{\m{A}\m{x} - \m{y}}^2, \st \rank\sbra{\cH\m{x}} = r. \label{eq:model0}
}
The insight underlying this reformulation is that the Hankel matrix $\cH\m{x}_\star$ of a spectrally sparse signal $\m{x}_\star$ has rank $r$ and is therefore low-rank when $r < \min\lbra{n_1,n_2}$ \cite{cai2018spectral,cai2019fast}. Conversely, by Kronecker's theorem \cite{ellis1992factorization}, a signal $\m{x}$ is spectrally sparse whenever its Hankel matrix $\cH\m{x}$ is low-rank, except for degenerate cases. 
This insight has led to extensive use of low-rank Hankel structures in spectral compressed sensing \cite{chen2014robust,andersson2014new,cai2016robust,cai2018spectral,cai2019fast,yang2024new,chen2022vectorized,li2024projected,cai2025accelerating}; see \cite{cai2026spectral} for a comprehensive overview.

In this paper, we study problem \eqref{eq:model0}. We propose a Lifted Restricted Isometry Property (LRIP), develop a Fast Iterative Hard Thresholding (FIHT) algorithm for efficient recovery, and establish exact recovery guarantees with near-optimal sample complexity. Our main contributions are summarized as follows:

\begin{itemize}
	\item For problem \eqref{eq:model0}, we propose a LRIP for the measurement matrix and prove that Gaussian measurements satisfy LRIP with $\mathcal{O}(r \ln^2 N)$ samples. Unlike the classical RIP \cite{candes2005decoding} for sparse vectors and the matrix RIP \cite{candes2011tight} for low-rank matrices, our LRIP imposes a rank constraint in the lifted matrix domain while taking measurements of the vectors obtained by applying the adjoint Hankel operator to these matrices (see Section \ref{sec:LRIP}).
	
	\item To efficiently solve \eqref{eq:model0}, we develop a FIHT algorithm with two initialization schemes: one-step hard thresholding and resampling-based FIHT (see Section \ref{sec:alg}).
	
	\item We establish exact recovery guarantees for FIHT based on the LRIP. When initialized by one-step hard thresholding, FIHT converges linearly to the true signal with high probability provided that the number of Gaussian measurements is $\mathcal{O}(r^2 \ln^2 N)$. When initialized by resampling, FIHT also converges linearly to the true signal with high probability under a reduced sample complexity of $\mathcal{O}(r \ln r \ln^2 N)$, which is near-optimal up to log factors and improves the state-of-the-art order of $r^2$ for nonconvex Hankel methods \cite{cai2019fast,cai2018spectral} (see Section \ref{sec:theory}).
\end{itemize}

In addition to the contributions above, we provide numerical results in Section \ref{sec:numeirical} to demonstrate the efficiency of FIHT. The proofs of the main results are given in Section \ref{sec:proof}, conclusion is presented in Section \ref{sec:conclusion}, and additional proofs are provided in the appendices.

\subsection{Related Work}
In \cite{cai2016robust}, Cai et al. consider the same spectral compressed sensing problem from linear compressive measurements and study Hankel matrix nuclear norm minimization, which serves a convex relaxation of problem \eqref{eq:model0}. They show that, as long as the number of Gaussian measurements exceeds $\cO\sbra{r \ln^2 N}$, the spectrally sparse signal can be exactly recovered with high probability. Note that a common drawback of nuclear norm minimization is its high computational complexity, typically of order $\cO\sbra{N^3}$ or higher for solving the equivalent semidefinite program (SDP). In contrast, our algorithm directly solves problem \eqref{eq:model0} with a significantly lower computational complexity that scales nearly linearly with $N$. Meanwhile, our sample complexity required for exact recovery is $\cO\sbra{r \ln r \ln^2 N}$, which matches the nuclear norm approach up to a log factor $\ln r$.

Heckel and Soltanolkotabi \cite{heckel2017generalized} show that $\cO(r \ln N)$ Gaussian measurements are sufficient to exactly recover spectrally sparse signals via atomic norm minimization (ANM) \cite{tang2013compressed,candes2014towards}, provided that the minimum wrap-around separation between frequencies is at least $4/N$. In contrast, our algorithm does not require any frequency separation for recovery guarantees and is significantly more computationally efficient than ANM.

Spectrally sparse signal recovery from partial entries is another instance of spectral compressed sensing. For this problem, nuclear norm minimization for low-rank Hankel matrix completion problem is proposed in \cite{chen2014robust}. Two fast algorithms, FIHT and projected gradient descent, are proposed in \cite{cai2019fast} and \cite{cai2018spectral}, respectively, to solve the fixed-rank Hankel completion problem. It is shown in \cite{cai2019fast,cai2018spectral} that the true signal can be exactly recovered under a sample complexity of order $r^2$ (up to log factors in $N$) and certain incoherence conditions. In contrast, we consider signal recovery from linear compressive measurements. Our theoretical guarantees rely on the LRIP, requiring only a sample complexity of order $r \ln r$, without any incoherence conditions. 

\subsection{Notation}
We denote by $\bR$ and $\bC$ the sets of real and complex numbers. Let $\mathbb{Z}$ denote the set of integers and $\mathbb{N}=\{1,2,\ldots\}$ the set of positive integers. Bold lowercase letters, bold uppercase letters, and calligraphic letters are reserved for vectors, matrices, and operators, respectively. 
For a vector $\m{x}$, $\twon{\m{x}}$ and $\diag\sbra{\m{x}}$ denote its $\ell_2$-norm and the diagonal matrix with $\m{x}$ on the main diagonal, respectively. 
For a matrix $\m{X}$, $\twon{\m{X}}$, $\frobn{\m{X}}$, and $\norm{\m{X}}_*$ denote its spectral norm, Frobenius norm, and nuclear norm, respectively. For an operator $\cA$, we define the operator norm as $\twon{\cA} := \sup_{\frobn{\m{X}}=1} \frobn{\cA\m{X}}$.
The identity operator is denoted by $\mathcal{I}$. 
We denote the adjoint of the Hankel operator $\cH$ by $\cH^*$, which maps a matrix $\m{X} \in \bC^{n_1 \times n_2}$ to a vector $\cH^*\m{X} = ( \sum_{i+j=a} X_{ij} )^{N-1}_{a=0}$. 
We use $\overline{(\cdot)}$, $(\cdot)^T$, and $(\cdot)^H$ to denote the complex conjugate, transpose, and conjugate transpose, respectively. 
Finally, $\mathcal{CN}(0,1)$ denotes the standard complex Gaussian distribution.

\section{Problem Formulation and LRIP} \label{sec:LRIP}
As seen from the Hankel matrix $\cH\m{x}$ in \eqref{eq:Hankel}, each entry of the signal $\m{x}$ appears a different number of times. This induces a non-uniform scaling of the signal energy, in the sense that $\frobn{\cH \m{x}} \neq \twon{\m{x}}$. Thus, $\cH$ is not an isometry, which complicates theoretical analyses.

To address this issue, we introduce a diagonal scaling operator
\equ{
	\cD := \diag \sbra{ \sqrt{w_0}, \ldots, \sqrt{w_{N-1}} },
}
where $w_a := \#\{(i,j)\mid i+j=a, \ i \in [0,n_1-1], \ j \in [ 0,n_2-1], \ i,j\in \bZ \}$ counts the multiplicity of each entry in the Hankel matrix. It follows that $\cD^2 = \cH^* \cH$.

We then define the scaled Hankel operator as
\equ{
	\cG := \cH \cD^{-1},
}
so that $\cG^*\cG = \cD^{-1} \cH^* \cH \cD^{-1} = \cI$.
That is, $\cG$ is an isometry and satisfies $\frobn{\cG \m{z}} = \twon{\m{z}}$.

Let $\m{z} := \cD\m{x}$ and $\m{B} := \m{A} \cD^{-1}$. We have $\m{B}\m{z} = \m{A}\m{x}$ and $\cG\m{z}=\cH\m{x}$. Therefore, problem \eqref{eq:model0} is equivalently reformulated as
\equ{
	\min_{\m{z}} \|\m{B}\m{z} - \m{y}\|_2^2, 
	\quad \st \rank\big(\cG \m{z}\big) = r. 
	\label{eq:model}
}
In addition, we denote $\m{z}_\star := \cD\m{x}_\star$. Then the measurement model can be rewritten as
\[
\m{y} = \m{A} \m{x}_\star = \m{B} \m{z}_\star.
\]

To establish recovery guarantees for \eqref{eq:model}, we rely on the RIP for the measurement matrix $\m{B}$, which is a fundamental concept in compressed sensing and low-rank matrix recovery \cite{foucart2013mathematical}.
In classical compressed sensing, the RIP requires the measurement matrix to approximately preserve the Euclidean norm of sparse vectors \cite{candes2005decoding}. 
This idea was later extended to low-rank matrix recovery via the matrix RIP, where the sensing operator is required to act nearly isometrically on the set of low-rank matrices \cite{candes2011tight}.

Establishing an analogous RIP for spectrally sparse signal recovery is substantially more challenging.
Unlike standard sparse vectors, spectrally sparse signals are parameterized by continuously-valued frequencies and therefore generally do not admit sparse representations under a fixed discrete dictionary that is universally valid for all such signals.

A related framework is the dictionary RIP (D-RIP) introduced in \cite{candes2011compressed}.
For signals sparse in a fixed overcomplete dictionary $\m D$, the D-RIP requires the measurement matrix to act nearly isometrically on the set $\lbra{ \m{D}\m{s}: \norm{\m{s}}_0 \le r }$, namely, the image of sparse vectors under $\m{D}$. However, this framework fundamentally relies on the existence of a fixed dictionary $\m D$ that is specified independently of the unknown signal.
In contrast, for spectrally sparse signals, the effective dictionary depends on the unknown continuously-valued frequencies and therefore varies with the signal itself.
Consequently, the D-RIP framework does not directly yield a universal recovery guarantee for spectrally sparse signals.

Recall that a spectrally sparse signal can be lifted into a low-rank Hankel matrix. 
This naturally motivates us to develop an RIP using the Hankel operator.
Specifically, for the recovery model in \eqref{eq:model}, stable recovery requires that the measurement matrix $\m B$ approximately preserves distances between feasible solutions, namely, $\twon{\m{B}\m{z}_1 - \m{B}\m{z}_2}^2 \approx \twon{ \m{z}_1 - \m{z}_2 }^2$, where $\m{z}_1$ and $\m{z}_2$ satisfy $\rank\big(\cG \m{z}_1\big) = \rank\big(\cG \m{z}_2\big) = r$. Using the identity $\cG^*\cG=\cI$, we have
\equ{
	\twon{\m{B}\m{z}_1 - \m{B}\m{z}_2}^2 
	=
	\twon{\m{B}\cG^* \sbra{ \cG\m{z}_1 - \cG\m{z}_2} }^2
	\quad \text{and} \quad
	\twon{ \m{z}_1 - \m{z}_2 }^2 = \twon{\cG^* \sbra{ \cG\m{z}_1 - \cG\m{z}_2} }^2.
}
Therefore, it suffices to establish
\equ{
	\twon{\m{B}\cG^* \sbra{ \cG\m{z}_1 - \cG\m{z}_2} }^2 
	\approx  
	\twon{\cG^* \sbra{ \cG\m{z}_1 - \cG\m{z}_2 } }^2.
}
Observe that $\cG\m{z}_1 - \cG\m{z}_2$ is both Hankel and low-rank.
Nevertheless, formulating an RIP over the set of low-rank Hankel matrices is overly restrictive, since intermediate quantities arising in recovery and convergence analysis generally preserve low-rankness but not exact Hankel structure.
Therefore, we instead impose the RIP over the broader class of low-rank matrices, which provides a more robust framework for stability and convergence analysis.

This motivates the following LRIP, where the rank constraint is imposed in the lifted matrix domain, whereas the measurements are taken after mapping these low-rank matrices back to the signal domain through $\cG^*$.

\begin{definition}[Lifted Restricted Isometry Property (LRIP)] \label{def:LRIP}
	For an integer $r\ge1$, define $\delta_r$ as the smallest nonnegative number such that
	\equ{
		\abs{ \twon{\m{B}\cG^*\m{Z}}^2-\twon{\cG^*\m{Z}}^2 }
		\le \delta_r\frobn{\m{Z}}^2
		\label{eq:LRIP_matrix}
	}
	for every $\m{Z}\in\bC^{n_1\times n_2}$ with $\rank(\m{Z})\le r$.
	If $\delta_r<1$, we say that $\m{B}$ satisfies LRIP at rank $r$.
	Here $\cG^*$ is the adjoint of the scaled Hankel operator and $N=n_1+n_2-1$.
\end{definition}

Since $\cG\cG^*$ is an orthogonal projector and $\frobn{\cG\cG^*\m{Z}}=\twon{\cG^*\m{Z}}$, we have $\frobn{\m{Z}}^2
=
\twon{\cG^*\m{Z}}^2
+
\frobn{(\cI-\cG\cG^*)\m{Z}}^2$. Therefore, \eqref{eq:LRIP_matrix} is equivalent to
\equ{
		(1-\delta_r)\frobn{\m{Z}}^2
		\le \twon{\m{B}\cG^*\m{Z}}^2
		+\frobn{(\cI-\cG\cG^*)\m{Z}}^2 \le (1+\delta_r)\frobn{\m{Z}}^2.
	\label{eq:LRIP_completed}
}
This formulation highlights the distinction from the standard low-rank RIP. The measurement matrix $\m{B}$ only acts on the signal-domain component $\cG^*\m{Z}$, while the component of $\m{Z}$ orthogonal to the Hankel subspace is not measured and must be accounted for separately.

In the following theorem, we show that if $\m{B}$ is composed of i.i.d.\ complex Gaussian entries, then $\m{B}$ satisfies the LRIP with high probability provided that the number of measurements is sufficiently large.

\begin{theorem} \label{thm:LRIP}
	Let $0 < \delta < 1$, $0 < \varepsilon < 1$, $N>1$, and $n_1,n_2 \in \bN$ satisfy $n_1+n_2-1 = N$ and $n_2 \ge n_1$. Denote $c_s := n_2/n_1$. Assume $\m{B}\in \bC^{m\times N}$ have entries that are i.i.d.\ $\cC\cN\sbra{0,1/m}$. Then, $\m{B}$ satisfies the LRIP with $\delta_r \le \delta$ with probability at least $1-\varepsilon$, provided
	\equ{
		m \ge C \delta^{-2} \mbra{ r \sbra{ c_s + \ln N } \ln N  + \ln \sbra{2\varepsilon^{-1}} }, \label{eq:m}
	}
	where $C>0$ is a universal constant. In particular, for $c_s = 1$, $\m{B}$ satisfies the LRIP with $\delta_r \le \delta$ with probability at least $1-N^{-\ln N}$, provided 
	\equ{
		m \ge C \delta^{-2} r \ln^2 N. 
	}
\end{theorem}
\begin{proof}
	The proof is deferred to Section \ref{subsec:proof_LRIP}. 
\end{proof}

\section{Algorithm} \label{sec:alg}
To efficiently solve \eqref{eq:model}, we develop a FIHT algorithm in Section \ref{subsec:FIHT} and introduce two initialization schemes in Section \ref{subsec:init}.

\subsection{FIHT Algorithm} \label{subsec:FIHT}
Problem \eqref{eq:model} is nonconvex due to the rank constraint, which necessitates the development of efficient algorithms with provable performance guarantees. Inspired by the iterative hard thresholding (IHT) algorithms for compressed sensing \cite{blumensath2009iterative} and Riemannian optimization methods for low-rank matrix recovery \cite{vandereycken2013low,wei2016guarantees,wei2020guarantees}, Cai et al. \cite{cai2019fast} propose a provable FIHT algorithm for low-rank Hankel matrix completion. FIHT is computationally efficient because by using subspace projections, it performs an singular value decomposition (SVD) on a much smaller matrix instead of the full Hankel matrix. 
Motivated by \cite{cai2019fast}, we develop an FIHT algorithm tailored to our problem \eqref{eq:model}, as shown in Algorithm \ref{alg:FIHT}. 

Our FIHT algorithm alternates between the vector space and the matrix space. Specifically, at each iteration, FIHT first updates the current estimate $\m{z}_t$ along the negative gradient direction $-\m{g}_t$. The resulting vector is then lifted to a Hankel matrix $\cG \sbra{ \m{z}_t - \m{g}_t }$ via the scaled Hankel operator $\cG$. Next, this Hankel matrix is projected onto the tangent space $T_t$ of the rank-$r$ matrix manifold at the previous iterate $\m{L}_t$ via the projection operator $\cP_{T_t}$, yielding an intermediate matrix $\m{W}_t$. The next iterate $\m{L}_{t+1}$ is obtained by applying the hard thresholding operator $\cT_r$ to $\m{W}_t$, which computes its best rank-$r$ approximation via truncated SVD. Finally,  the new estimate $\m{z}_{t+1}$ is obtained by applying the adjoint operator $\cG^*$ to $\m{L}_{t+1}$.

\begin{algorithm}
	\caption{Fast Iterative Hard Thresholding (FIHT) for solving \eqref{eq:model}}
	\label{alg:FIHT}
	\begin{algorithmic}
		\REQUIRE Observations $\m{y}$, measurement matrix $\m{B}$, and sparsity level $r$ \\
		Initialize a rank-$r$ matrix $\m{L}_0$ and set $\m{z}_0=\cG^*\m{L}_0$
		\FOR{$t = 0,1,...$}
		\STATE $\m{g}_t = \m{B}^H\sbra{\m{B}\m{z}_t - \m{y}}$ \
		\STATE $\m{W}_t = \cP_{T_t} \cG \sbra{ \m{z}_t -  \m{g}_t }$ \
		\STATE $\m{L}_{t+1} = \cT_r\sbra{\m{W}_t}$ \
		\STATE $\m{z}_{t+1} = \cG^* \m{L}_{t+1}$
		\ENDFOR
		\ENSURE $\m{z}_{t+1}$
	\end{algorithmic}
\end{algorithm}

FIHT is both computationally and memory efficient. Specifically, the Hankel matrix-matrix multiplications and the adjoint Hankel operations involved in each iteration can be efficiently implemented via the fast Fourier transform (FFT), requiring only $\cO\sbra{r N \log N}$ flops \cite{cai2019fast}. Moreover, the truncated SVD of the $n_1 \times n_2$ matrix $\m{W}_t$ can be computed from the SVD of a much smaller $2r\times 2r$ matrix, requiring only $\cO\sbra{r^3}$ flops. Consequently, the overall per-iteration computational complexity of FIHT is $\cO\sbra{ m N + r^2 N + r N \log N + r^3 }$, where the $mN$ and $r^2 N$ terms arise from the measurement matrix-vector and matrix-matrix multiplications, respectively. By substituting the theoretical sample complexity lower bound $m=\cO\sbra{r \ln^2 N}$ in Theorem \ref{thm:LRIP} and noting that $r \ll N$, the complexity becomes $\cO\sbra{ r N \log^2 N + r^2 N }$, which scales nearly linearly with $N$.
In terms of memory requirements, except for the $m \times N$ measurement matrix $\m{B}$, the other largest matrices that need to be stored are the singular vector matrices of $\m{W}_t$, which require only $\cO\sbra{rN}$ memory. Implementation details of FIHT are deferred to Appendix \ref{append:FIHT}.

\subsection{Initialization} \label{subsec:init}
As is common in nonconvex optimization, the recovery performance of FIHT is dependent on the initialization. We next describe two initialization schemes.

\subsubsection{Initialization via One-Step Hard Thresholding}
The initial estimate $\m{L}_0$ is obtained by truncating the Hankel matrix constructed from $\m{B}^H \m{y}$. Formally,
\equ{
	\m{L}_0 = \cT_r\sbra{ \cG \m{B}^H \m{y} }.
} 
This initialization is easy to implement; however, it typically yields only a suboptimal sample complexity for FIHT, which scales quadratically with the sparsity level $r$, as we will show in Theorem \ref{thm:guarantee1}.

\subsubsection{Initialization via Resampling}
To improve the sample complexity bound, we propose a resampled FIHT for initialization in Algorithm \ref{alg:Init-resamp}. We set the initial estimate 
\equ{
	\m{L}_0=\widetilde{\m{L}}_T,
}
where $\widetilde{\m{L}}_T$ is the output of Algorithm \ref{alg:Init-resamp}.

For a prescribed integer $T\ge0$ with $m/(T+1)\in\bN$, Algorithm \ref{alg:Init-resamp} begins by partitioning the measurement index set $\lbra{1,\ldots,m}$ into $T+1$ disjoint subsets $\Omega_0, \ldots, \Omega_T$, each of size $\widetilde{m}$. Let $\m{B} := \mbra{\m{b}_1,\ldots,\m{b}_m}^H$. For each subset, define the rescaled measurement matrix and observations by
\equ{
	\m{B}_t := \sqrt{\frac{m}{\widetilde{m}}} \mbra{\m{b}_i^H}_{i\in \Omega_t}, \qquad \m{y}_t := \sqrt{\frac{m}{\widetilde{m}}} \mbra{y_i}_{i\in \Omega_t}.
}
It follows that $\m{y}_t = \m{B}_t \m{z}_\star$.
At iteration $t$, a refined estimate is computed by applying FIHT to the fresh measurements $(\m{B}_{t+1}, \m{y}_{t+1})$. This resampling step removes the statistical dependence between the current iterate $\widetilde{\m{L}}_t$ and the measurements used in the update, which enables a sharp spectral concentration bound (Lemma \ref{lem:KeySpectralNormBound}). As a result, the sample complexity is near-optimal in $r$ up to a log factor; see Theorem \ref{thm:guarantee2}.

\begin{algorithm}
	\caption{Initialization via Resampled FIHT}
	\label{alg:Init-resamp}
	\begin{algorithmic}
		\REQUIRE Observations $\m{y}$, measurement matrix $\m{B}$, and sparsity level $r$ \\
		Partition $\lbra{1,\ldots,m}$ into $T+1$ disjoint subsets $\Omega_0, \ldots, \Omega_T$ of equal size $\widetilde{m}$ \\
		Initialize $\widetilde{\m{L}}_0 = \cT_r\sbra{ \cG \m{B}_0^H \m{y}_0 }$ and set $\widetilde{\m{z}}_0 = \cG^* \widetilde{\m{L}}_0$
		\FOR{$t = 0,1,...,T-1$}
			\STATE $\widetilde{\m{g}}_t = \m{B}_{t+1}^H\sbra{\m{B}_{t+1}\widetilde{\m{z}}_t - \m{y}_{t+1}}$ \
			\STATE $\widetilde{\m{W}}_t = \cP_{\widetilde{T}_t} \cG \sbra{ \widetilde{\m{z}}_t -  \widetilde{\m{g}}_t }$ \
			\STATE $\widetilde{\m{L}}_{t+1} = \cT_r\sbra{\widetilde{\m{W}}_t}$ \
			\STATE $\widetilde{\m{z}}_{t+1} = \cG^* \widetilde{\m{L}}_{t+1}$
		\ENDFOR
		\ENSURE $\widetilde{\m{L}}_T$
	\end{algorithmic}
\end{algorithm}

\begin{rem}
	The resampling scheme is inspired by \cite{cai2019fast}, but our purposes are different. In \cite{cai2019fast}, for the low-rank Hankel completion problem, resampling is used to remove the dependence of the sample complexity on $\sqrt{N}$, which does not arise in our problem. In contrast, our resampling aims to reduce the dependence on $r^2$ to $r$ in the sample complexity. As a result, our subsequent theoretical analysis differs significantly from that in \cite{cai2019fast}.
\end{rem}

\section{Recovery Guarantees} \label{sec:theory}
We first establish a recovery guarantee for FIHT with one-step hard thresholding initialization.

\begin{theorem} \label{thm:guarantee1}
	Assume $\m{B}\in \bC^{m\times N}$ satisfies the LRIP with constant $\delta_{4r} \le \frac{1}{136\sqrt{r}\kappa}$, where $\kappa$ denotes the condition number of $\cG\m{z}_\star$. Let the initialization be $\m{L}_0 = \cT_r\sbra{ \cG \m{B}^H \m{y} }$. Then the FIHT iterates  satisfy
	\equ{
		\twon{\m{z}_t - \m{z}_\star} \le v^t \frobn{\m{L}_0 - \cG\m{z}_\star},
	}
	for some constant $v<1$.
	
	Moreover, if $\m{B}$ has i.i.d.\ entries following $\cC\cN\sbra{0,1/m}$, then with probability at least $1 - N^{-\sbra{c_s + \ln N}}$, the condition $\delta_{4r} \le \frac{1}{136\sqrt{r}\kappa}$ holds provided that
	\equ{
		m \ge C \kappa^2 r^2 \max \lbra{c_s, \ln N} \ln N,
		\label{eq:m_guarantee1}
	}
	for some universal constant $C>0$.
\end{theorem}

\begin{proof}
	The proof is deferred to Section \ref{sec:recovery1}.
\end{proof}

We next provide a refined guarantee under resampling-based initialization.

\begin{theorem} \label{thm:guarantee2}
	Assume $r\ge 2$ and $\m{B} \in \bC^{m\times N}$ has i.i.d.\ entries following $\cC\cN\sbra{0,1/m}$, and let $T = \lceil \ln r \rceil - 1$ with $m/(T+1)\in \bN$. Let the initialization be $\m{L}_0 = \widetilde{\m{L}}_T$. Then, with probability at least $1 - \sbra{3+\ln r} N^{-\sbra{c_s + \ln N}}$, the FIHT iterates satisfy
	\equ{
		\twon{\m{z}_t - \m{z}_\star} \le v^t \frobn{\m{L}_0 - \cG\m{z}_\star}, \label{eq:zz}
	}
	for some constant $v < 1$, provided that
	\equ{
		m \ge C \kappa^2 r \ln r \max\lbra{ c_s, \ln N } \ln N,
		\label{eq:m_guarantee2}
	}
	where $C>0$ is a universal constant. For $r=1$, the same convergence and probability bounds hold with $T=0$ provided that $m\ge C\kappa^2\max\{c_s,\ln N\}\ln N$.
\end{theorem}

\begin{proof}
	The proof is deferred to Section \ref{sec:recovery2}.
\end{proof}

Both theorems establish exact recovery with linear convergence of FIHT under different initialization schemes. Comparing \eqref{eq:m_guarantee1} and \eqref{eq:m_guarantee2}, the sample complexity improves from $\mathcal{O}(r^2)$ to $\mathcal{O}(r \ln r)$ in terms of its dependence on $r$. This rate is near-optimal up to logarithmic factors, as an $r$-spectrally-sparse signal has $\mathcal{O}(r)$ intrinsic degrees of freedom.
The improvement stems from a sharper initialization error bound achieved via resampling, as shown by comparing \eqref{eq:L0gz} and \eqref{eq:L0gz2}.

\section{Numerical Results} \label{sec:numeirical}
In this section, we present numerical experiments to evaluate the performance of FIHT and compare it with Hankel nuclear norm minimization (Hankel-NNM) \cite{cai2016robust} and atomic norm minimization (ANM) \cite{heckel2017generalized}. Specifically, FIHT is terminated when the relative change satisfies $\norm{\m{z}_{t+1} - \m{z}_t }_2 / \norm{\m{z}_t}_2 \le 10^{-8}$ or when the maximum number of iterations $1000$ is reached.We use an adaptive step size to improve practical performance. For FIHT initialized by resampling in Algorithm \ref{alg:Init-resamp}, we choose $T = \min \left\{ s \ge \max\lbra{\lceil \ln r \rceil,1}-1 : \frac{m}{s+1} \in \mathbb{N} \right\} $.
For Hankel-NNM and ANM, we solve the corresponding optimization problems using the interior-point solver SDPT3 \cite{toh2012implementation} within the CVX toolbox \cite{cvx}. 

\begin{figure*}[htb] 
	\begin{center}
		\subfigure[FIHT (one-step initialization)]%
		{\includegraphics[width=.40\textwidth]{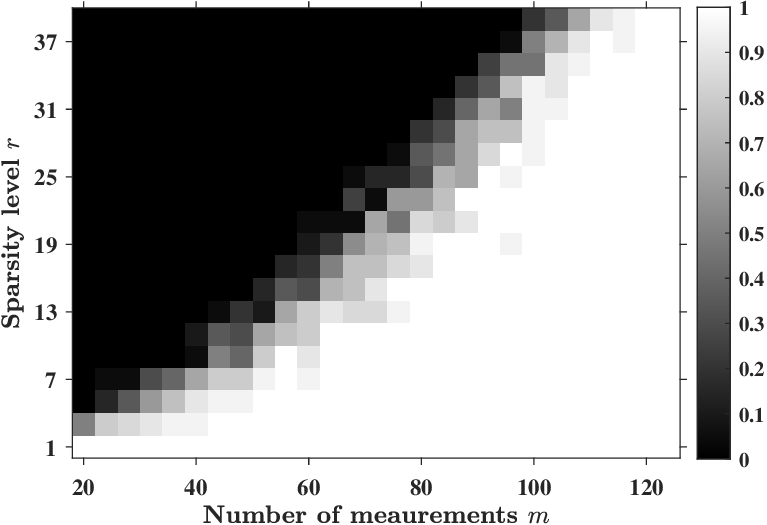}}
		\qquad
		\subfigure[FIHT (resampling initialization)]%
		{\includegraphics[width=.40\textwidth]{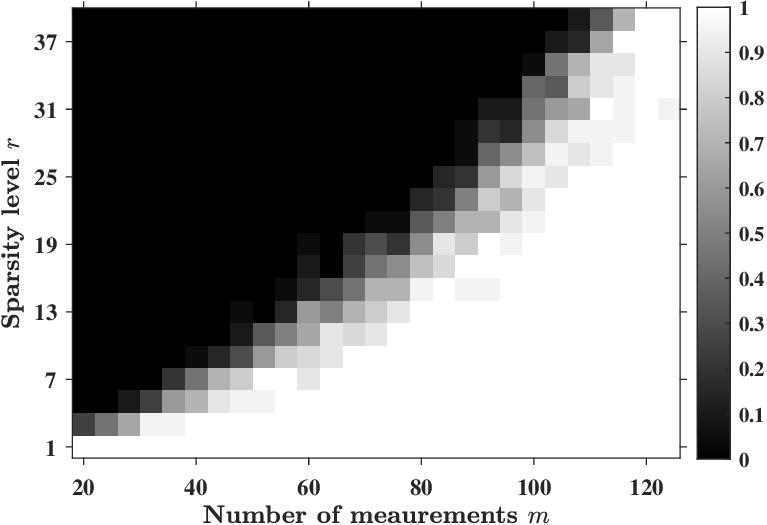}}
		\subfigure[Hankel-NNM \cite{cai2016robust}]%
		{\includegraphics[width=.40\textwidth]{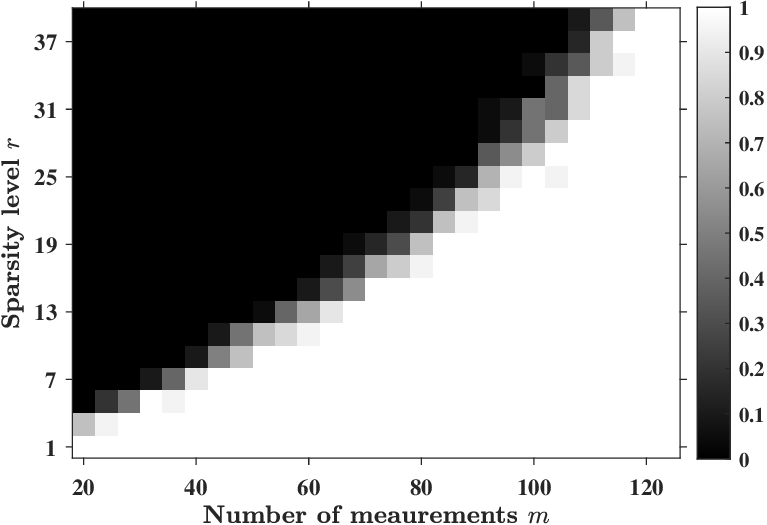}}
		\qquad
		\subfigure[ANM \cite{heckel2017generalized}]%
		{\includegraphics[width=.40\textwidth]{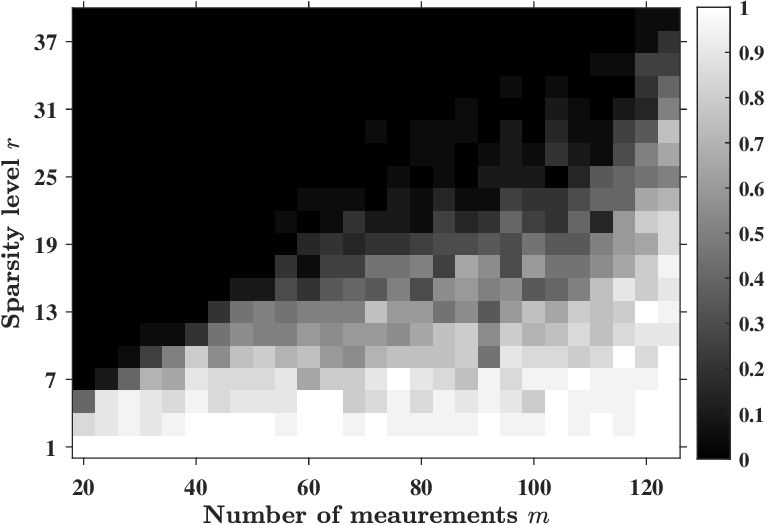}}
	\end{center}
	\caption{Phase transition comparisons. White means complete success and black means complete failure.} \label{fig:exp1}
\end{figure*}

\subsection{Phase Transition} \label{sec:phase}
We study the signal recovery performance of FIHT through phase transition experiments.
We first generate spectrally sparse signals $\m{x}_\star$ of length $N=127$ with $r$ frequency components. Each frequency $f_k$ is sampled uniformly from $[0,1)$, the damping factor $\tau_k$ is set to $0$, and the amplitude ${s_k}$ is generated as $s_k = \sbra{1+10^{0.5 m_k}}e^{i 2\pi \phi_k}$, where $m_k$ and $\phi_k$ are independently sampled from the uniform distribution on $[0,1]$. The Hankel matrix size is chosen as $n_1=n_2=(N+1)/2=64$.
We then obtain $m$ measurements $\m{y}$ of $\m{x}_\star$ via random Gaussian sampling, $\m{y}=\m{B}\cD\m{x}_\star$, where the entries of $\m{B}$ follow $\cC\cN(0,1/m)$ and $\cD$ is defined in Section~\ref{sec:LRIP}.
A signal is regarded as successfully recovered if the root mean squared error (RMSE) satisfies $\norm{\widehat{\m{x}} - \m{x}_\star }_2 / \norm{\m{x}_\star}_2 \le 10^{-3}$, where $\widehat{\m{x}}$ denotes the recovered signal. The success rate is estimated from $20$ Monte Carlo trials for each pair $(m,r)$, with $m\in\lbra{20,24,\ldots,124}$ and $r\in \lbra{1,3,\ldots,39}$.

The resulting phase transition is shown in Figure~\ref{fig:exp1}. We observe that FIHT exhibits comparable performance to Hankel-NNM under both one-step and resampling initialization schemes, while significantly outperforming ANM. The inferior performance of ANM is expected since it requires a minimum separation condition on the frequencies \cite{yang2024separation}.

\begin{figure*}[htb] 
	\centering
	\subfigure {\includegraphics[width=6.3cm]{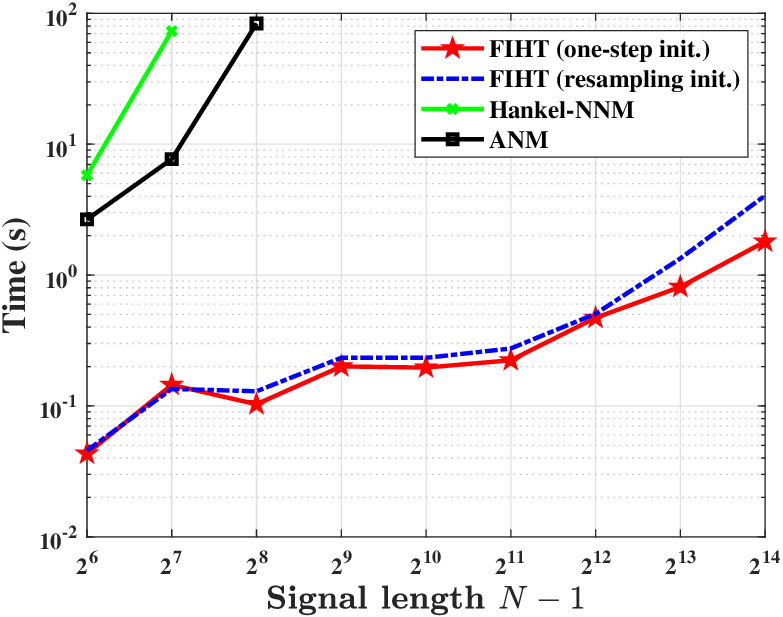} }
	\qquad
	\subfigure {\includegraphics[width=6.3cm]{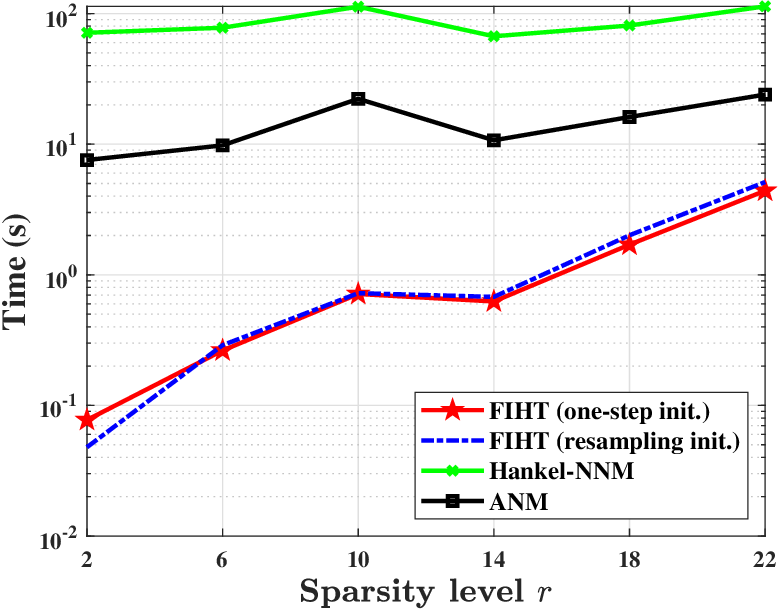}  }
	\caption{\emph{Left}: Computational time vs signal length $N$. \emph{Right}: Computational time vs sparsity level $r$.}
	\label{fig:time}
\end{figure*}

\begin{figure*}[htb] 
	\centering
	\subfigure[$m=80$]{\includegraphics[width=6.3cm]{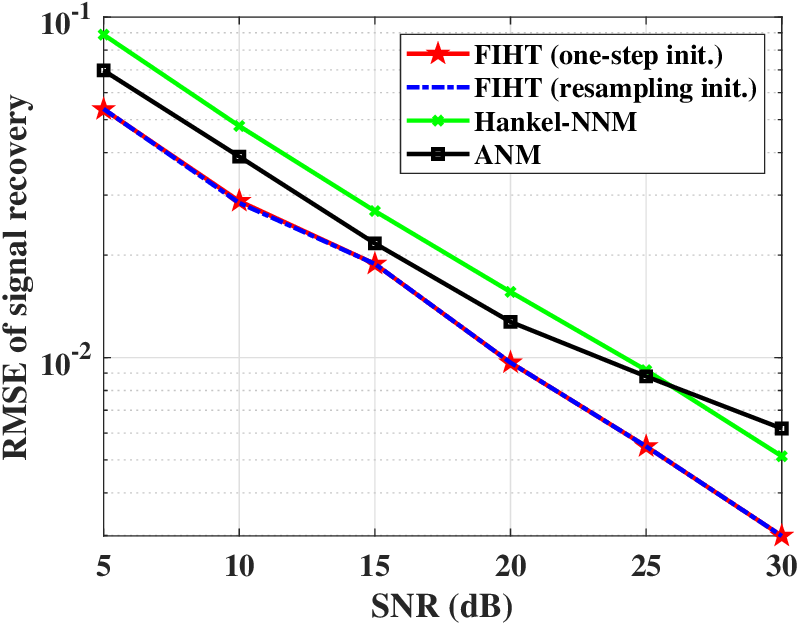} } 
	\qquad
	\subfigure[$m=100$]{\includegraphics[width=6.3cm]{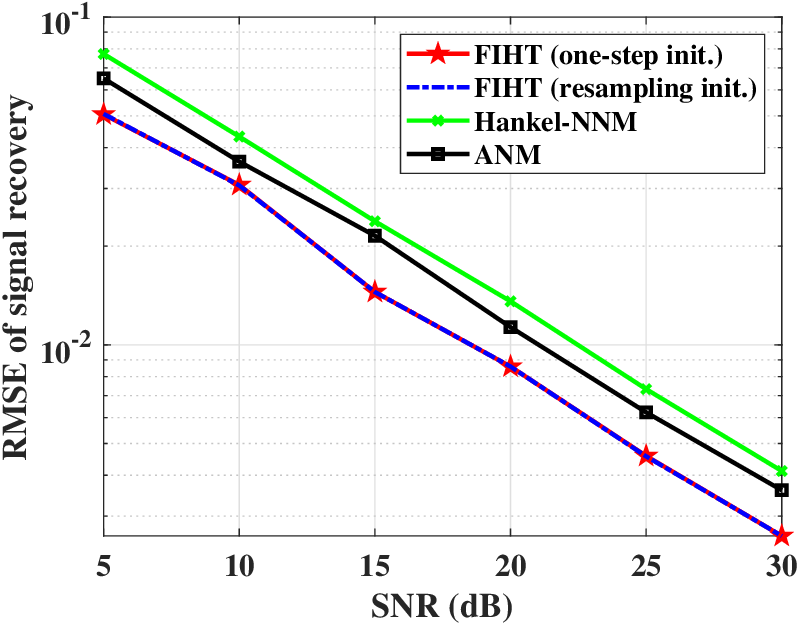} }
	\caption{RMSE of signal recovery versus SNR.} 
	\label{fig:Noisy}
\end{figure*}

\subsection{Computational Time}
We evaluate the computational efficiency of FIHT under different signal lengths and sparsity levels. 
The signals are generated in the same manner as in Section~\ref{sec:phase}.
We first fix $r=3$ as well as $m=\lfloor 0.8N \rfloor$, and vary the signal length $N=2^j+1$ for $j\in\lbra{6,7,\ldots,14}$. 
We then fix $N=127$ as well as $m=100$, and vary the sparsity level $r$ over the range $\lbra{2,6,\ldots,22}$.

Figure~\ref{fig:time} reports the average computational time over $50$ Monte Carlo trials in which all compared methods successfully recover the signal. 
As shown in the figure, FIHT is significantly faster than Hankel-NNM and ANM due to its lower computational complexity.

\subsection{Robustness to Noise}
We evaluate the robustness of FIHT in the presence of additive noise. 
The signal length $N$ is set to $127$. 
We randomly generate $r=5$ frequencies and add complex white Gaussian noise $\m{\xi}$ to the measurements $\m{y}$. 
The signal-to-noise ratio (SNR) is defined as $\text{SNR} =10\log_{10} (\twon{ \m{y} }^2/\twon{\m{\xi}}^2)$.
We vary the SNR from $5$ to $30$ dB.
For Hankel-NNM and ANM, we impose an inequality constraint to ensure that the data-fidelity error does not exceed the known noise energy. 
Figure~\ref{fig:Noisy} shows the RMSE of the recovered signal averaged over $50$ trials when $m=80$ and $m=100$. 
It can be observed that FIHT achieves the lowest estimation error.

\section{Proofs} \label{sec:proof}

\subsection{Proof of Theorem \ref{thm:LRIP}} \label{subsec:proof_LRIP}
The basic idea of the proof is to reformulate the LRIP analysis as the study of the supremum of a quadratic Gaussian chaos process. 
To control this quantity, we apply the deviation inequality in \cite[Theorem 3.1]{krahmer2014suprema}, restated in Lemma~\ref{lem:chaos}. 
The key quantity in this result is Talagrand's $\gamma_2$ functional, which can be controlled by the Gaussian width in Lemma \ref{lem:Talagrand}. 
Consequently, the main technical step is to establish a sharp Gaussian width bound for the structured set in LRIP, as given in Lemma~\ref{lem:width}. 
This argument relies on a new spectral norm estimate for random rectangular Hankel matrices established in Lemma~\ref{lem:norm}.

Before stating Lemma \ref{lem:chaos}, we introduce two complexity measures associated with a set of matrices $\cA$.
The first is the radius of $\cA$ with respect to the Frobenius norm and the spectral norm, defined as
$d_F\sbra{\cA} := \sup_{\m{A}\in \cA} \frobn{\m{A}}$ and $d_2\sbra{\cA} := \sup_{\m{A}\in \cA} \twon{\m{A}}$. The second is Talagrand's $\gamma_2$ functional $\gamma_2 \sbra{\cA, d}$, where $d$ is a metric; see \cite{talagrand2005generic} for its definition.

\begin{lemma}[{\cite[Theorem 3.1]{krahmer2014suprema}}] \label{lem:chaos}
	Let $\cA$ be a set of matrices, and $\m{\xi}$ be a random vector whose entries are i.i.d.\ $\cC\cN\sbra{0,1}$. Set
	\equ{
		\begin{split}
			E & := \gamma_2 \sbra{\cA, \norm{\cdot}_2} \mbra{ \gamma_2 \sbra{\cA, \norm{\cdot}_2} + d_F\sbra{\cA} } + d_F\sbra{\cA} d_2 \sbra{\cA}, \\
			V & := d_2 \sbra{\cA} \mbra{ \gamma_2 \sbra{\cA, \norm{\cdot}_2} + d_F\sbra{\cA} }, \\
			U & := d^2_2 \sbra{\cA}.
		\end{split} \label{eq:EVU}
	} 
	Then, for any $t>0$, 
	\equ{
		\bP \sbra{ \sup_{\m{A}\in \cA} \abs{ \twon{\m{A}\m{\xi}}^2 - \bE \twon{\m{A}\m{\xi}}^2 } \ge c_1 E + t } \le 2 \exp \sbra{ -c_2 \min \lbra{ \frac{t^2}{V^2}, \frac{t}{U} } },
	}
	where $c_1,c_2 >0$ are absolute constants.
\end{lemma}

\begin{lemma}[{\cite[Theorem 2.1.1]{talagrand2005generic}, Talagrand's majorizing measure theorem}] \label{lem:Talagrand}
	Let $\lbra{G_t: t\in \cA}$ be a mean-zero Gaussian process indexed by a set $\cA$ and $d \sbra{t_1, t_2}:= (\bE \abs{ G_{t_1} - G_{t_2} }^2)^{\frac{1}{2}}$ be the canonical metric induced by the process. Then, 
	\equ{
		\gamma_2 \sbra{\cA,d} \le C \cdot \bE \sup_{t\in \cA} G_t,
	}	 
	where $C>0$ is an absolute constant.
\end{lemma}

\begin{lemma} \label{lem:width}
	Let $S_r := \lbra{ \cG^*\m{Z}: \rank\sbra{\m{Z}} \le r, \frobn{\m{Z}} = 1 } \subset \bC^N$ and $\m{g} \in \bC^N$ be a random vector whose entries are i.i.d.\ $\cC\cN\sbra{0,1}$. Then, for $\m{z}\in \bC^N$, $\Re{ \left \langle \m{g}, \m{z} \right \rangle }$ is a mean-zero Gaussian process. And, the Gaussian width of $S_r$ is defined by $\omega \sbra{S_r} := \bE_{\m{g}} \sup_{\m{z}\in S_r} \Re{ \left \langle \m{g}, \m{z} \right \rangle } $ satisfying
	\equ{
		\omega \sbra{S_r} \le 2 \sqrt{ r \sbra{ c_s + \ln N } \ln N  }.
	}
\end{lemma}

\begin{proof}
	Denote $\widetilde{\m{g}}:=\sbra{\Re\sbra{\m{g}}; \Im\sbra{\m{g}}} \in \bR^{2N}$ and $\widetilde{\m{z}}:=\sbra{\Re\sbra{\m{z}}; \Im\sbra{\m{z}}} \in \bR^{2N}$. We have $\Re{ \left \langle \m{g}, \m{z} \right \rangle } = \left \langle \widetilde{\m{g}}, \widetilde{\m{z}} \right \rangle$. Since $ \left \langle \widetilde{\m{g}}, \widetilde{\m{z}} \right \rangle$ is a Gaussian process \cite{vershynin2018high}, then $\Re{ \left \langle \m{g}, \m{z} \right \rangle }$ is also a Gaussian process. Moreover,
	\equ{
		\begin{split}
			\omega \sbra{S_r}
			& \le \bE_{\m{g}} \sup_{\m{z}\in S_r} \abs{ \left \langle \m{g}, \m{z} \right \rangle } \\
			& = \bE_{\m{g}} \sup_{ \rank  \sbra{\m{Z}} \le r, \frobn{\m{Z}} = 1 } \abs{ \left \langle \m{g}, \cG^*\m{Z} \right \rangle } \\
			& \le \bE_{\m{g}} \sup_{ \rank  \sbra{\m{Z}} \le r, \frobn{\m{Z}} = 1 } \twon{\cG \m{g}} \norm{\m{Z}}_* \\
			& \le \sqrt{r} \bE_{\m{g}} \twon{\cG \m{g}} \\
			& \le 2 \sqrt{ r \sbra{ c_s + \ln N } \ln N  }, 
		\end{split}
	}
	where the third inequality follows from $\norm{\m{Z}}_* \le \sqrt{\rank\sbra{\m{Z}}} \frobn{\m{Z}}$ and the last one follows from Lemma \ref{lem:norm}.
\end{proof}

\begin{lemma} \label{lem:norm}
	Let $\m{g} \in \bC^N$ be a random vector whose entries are i.i.d.\ $\cC\cN\sbra{0,1}$ and $\cG\m{g} \in \bC^{n_1 \times n_2}$ denote the associated Hankel matrix, where $n_1+n_2-1 = N$ and $N>1$. Denote $c_s := n_2/n_1 \ge 1$. Then,
	\equ{
		\bE \twon{\cG\m{g}} \le 2 \sqrt{ \sbra{ c_s + \ln N } \ln N }.
	}
	When $n_1 = n_2$, we further have $\bE \twon{\cG\m{g}} \le 4 \ln N$.
\end{lemma}

\begin{proof}
	The proof of Lemma \ref{lem:norm} is deferred to Appendix \ref{append:norm}.
\end{proof}

\begin{rem}
	Lemma~\ref{lem:norm} extends \cite[Theorem~3]{cai2016robust} to rectangular Hankel matrices and quantifies how the size $\sbra{n_1,n_2}$ affects the upper bound, with an explicit constant.
	In addition, using \cite[Theorem~4.1.1]{tropp2015introduction} (restated in Lemma \ref{lem:trop}), we provide a more direct proof than that in \cite{cai2016robust}.
\end{rem}

\begin{proof}[Proof of Theorem \ref{thm:LRIP}]
	Recall that $S_r = \lbra{ \cG^*\m{Z}: \rank\sbra{\m{Z}} \le r, \frobn{\m{Z}} = 1 }$. For $\m{z} \in S_r$, we define 
	\equ{
		\m{V}_{\m{z}} := \frac{1}{\sqrt{m}} \m{I}_m \otimes \m{z}^T \in \bC^{m \times mN}, \qquad \m{\xi} := \sqrt{m} \text{vec} \sbra{\m{B}^T} \in \bC^{mN}.  \label{eq:Vz}
	}
	Here, $\otimes$ is the Kronecker product and $\text{vec} \sbra{\cdot}$ is the vectorization operator that stacks the columns of a matrix into a vector.
	Then, we have
	\equ{
		\m{B}\m{z} = \text{vec}\sbra{ \m{B}\m{z} } = \text{vec}\sbra{ \m{z}^T\m{B}^T \m{I}_m } = \sbra{ \m{I}_m \otimes \m{z}^T} \text{vec} \sbra{\m{B}^T} = \m{V}_{\m{z}} \m{\xi}. \label{eq:Bz=Vg}
	}
	
	Let
	\equ{
		\cA := \lbra{\m{V}_{\m{z}}: \m{z}\in S_r}. \label{eq:A}
	}
	Then, we have
	\equ{
		\begin{split}
			\delta_r
			& =\sup_{\substack{\rank(\m{Z})\le r\\\m{Z}\ne\m{0}}}
			\frac{\abs{\twon{\m{B}\cG^*\m{Z}}^2-\twon{\cG^*\m{Z}}^2}}{\frobn{\m{Z}}^2}\\
			& =\sup_{\substack{\rank(\m{Z})\le r\\\frobn{\m{Z}}=1}}
			\abs{\twon{\m{B}\cG^*\m{Z}}^2-\twon{\cG^*\m{Z}}^2}\\
			& =\sup_{\m{z}\in S_r}\abs{\twon{\m{B}\m{z}}^2-\twon{\m{z}}^2} \\
			& =\sup_{\m{V}_{\m{z}}\in\cA}
			\abs{\twon{\m{V}_{\m{z}}\m{\xi}}^2-\bE\twon{\m{V}_{\m{z}}\m{\xi}}^2},
		\end{split} \label{eq:delta_r}
	}
	where the last equality follows from \eqref{eq:Bz=Vg} and $\bE\twon{\m{B}\m{z}}^2=\twon{\m{z}}^2$ since the entries of $\m B$ are i.i.d.\ $\mathcal{CN}(0,1/m)$.
	
	Since we aim to show $\bP \sbra{\delta_r \le \delta} \ge 1 - \varepsilon$, it follows from \eqref{eq:delta_r} that it suffices to show 
	\equ{
		\bP\sbra{ \sup_{\m{V}_{\m{z}} \in \cA} \abs{ \twon{\m{V}_{\m{z}} \m{\xi} }^2 - \bE \twon{\m{V}_{\m{z}} \m{\xi} }^2 } \le \delta } \ge 1-\varepsilon,
	}
	or equivalently,
	\equ{
		\bP\sbra{ \sup_{\m{V}_{\m{z}} \in \cA} \abs{ \twon{\m{V}_{\m{z}} \m{\xi} }^2 - \bE \twon{\m{V}_{\m{z}} \m{\xi} }^2 } \ge \delta } \le \varepsilon.
	}
	Since the entries of $\m B$ are i.i.d.\ $\mathcal{CN}(0,1/m)$, it follows from \eqref{eq:Vz} that $\m{\xi}$ has i.i.d.\ entries following $\cC\cN\sbra{0,1}$. Then applying Lemma \ref{lem:chaos}, it suffices to show that $c_1 E + t \le \delta$ and $2 \exp \sbra{ -c_2 \min \lbra{ \frac{t^2}{V^2}, \frac{t}{U} } } \le \varepsilon$. Setting $t = \delta / 2$, the above conditions reduce to 
	\equ{
		E \le \frac{\delta}{2c_1}, \qquad V \le \frac{\delta\sqrt{c_2}}{2}  \sbra{\ln \frac{2}{\varepsilon}}^{-1/2}, \qquad U\le \frac{\delta c_2}{2} \sbra{\ln \frac{2}{\varepsilon}}^{-1}. \label{eq:EVU_condtion}
	}
	
	It remains to compute $E$, $V$, and $U$ in \eqref{eq:EVU} with respect to $\cA$. Specifically, we have
	\equ{
		\begin{split}
			d_F\sbra{\cA} & = \sup_{\m{V}_{\m{z}}\in{\cA}} \frobn{\m{V}_{\m{z}}} = \sup_{\m{z}\in S_r} \frac{1}{\sqrt{m}} \frobn{ \m{I}_m \otimes \m{z}^T } = \sup_{\m{z}\in S_r} \frac{1}{\sqrt{m}} \sqrt{m} \twon{\m{z}} \le 1, \\
			d_2\sbra{\cA} & = \sup_{\m{V}_{\m{z}}\in{\cA}} \twon{\m{V}_{\m{z}}} = \sup_{\m{z}\in S_r} \frac{1}{\sqrt{m}} \twon{ \m{I}_m \otimes \m{z}^T } = \sup_{\m{z}\in S_r} \frac{1}{\sqrt{m}} \twon{\m{I}_m} \twon{\m{z}} \le \frac{1}{\sqrt{m}},
		\end{split} \label{eq:dFd2}
	}
	where we used the property $\twon{ \m{A} \otimes \m{B}} = \twon{\m{A}} \twon{\m{B}}$.
	
	Note that for $\m{V}_{\m{z}_i} \in \cA$ and $\m{z}_i \in S_r$, $i = 1,2$, we have 
	\equ{
	\twon{\m{V}_{\m{z}_1}-\m{V}_{\m{z}_2}}
	= \twon{ \frac{1}{\sqrt m} \sbra{ \m{I}_m \otimes (\m{z}_1-\m{z}_2)^T } }
	= \frac{1}{\sqrt m} \twon{\m{z}_1-\m{z}_2}.
	}
	Combining this with the homogeneous property of the $\gamma_2$ functional yields
	\equ{
		\gamma_2 \sbra{\cA, \twon{\cdot}} 
		= \frac{1}{\sqrt{m}} \gamma_2 \sbra{S_r, \twon{\cdot}}. \label{eq:gamma2=}
	}
	To apply Lemma~\ref{lem:Talagrand}, we need construct a Gaussian process associated with $S_r$ under the metric $\twon{\cdot}$.
	Let $\widetilde{\m{z}}:=\sbra{\Re(\m{z}); \Im(\m{z})}\in \bR^{2N}$ and $\widetilde{S}_r := \lbra{ \widetilde{\m{z}}: \m{z}\in S_r }$. Let $\m{g} \in \bC^N$ with i.i.d.\ entries following $\cC\cN\sbra{0,1}$ and $\widetilde{\m{g}} := \sbra{\Re(\m{g}); \Im(\m{g})} \in \bR^{2N}$ with entries following $\cN\sbra{0,1/2}$. We consider the Gaussian process $\left \langle \widetilde{\m{g}}, \widetilde{\m{z}} \right \rangle$.
	For $\widetilde{\m{z}}_1,\widetilde{\m{z}}_2 \in \widetilde{S}_r$, the metric induced by this process satisfies
	\equ{
		d^2\sbra{ \widetilde{\m{z}}_1, \widetilde{\m{z}}_2 } = \bE \abs{ \left \langle \widetilde{\m{g}}, \widetilde{\m{z}}_1 \right \rangle - \left \langle \widetilde{\m{g}}, \widetilde{\m{z}}_2 \right \rangle }^2 = \sbra{ \widetilde{\m{z}}_1 - \widetilde{\m{z}}_2 }^T \bE \sbra{\widetilde{\m{g}} \widetilde{\m{g}}^T } \sbra{ \widetilde{\m{z}}_1 - \widetilde{\m{z}}_2 } = \frac{1}{2} \twon{\widetilde{\m{z}}_1 - \widetilde{\m{z}}_2}^2, \label{eq:metric}
	} 
	corresponding to the metric $\twon{\cdot}$.
	Then we have
	\equ{
		\gamma_2 \sbra{\cA, \twon{\cdot}}
		= \frac{1}{\sqrt{m}} \gamma_2 \sbra{\widetilde{S}_r, \twon{\cdot}} 
		\le \frac{C}{\sqrt{m}} \bE_{\widetilde{\m{g}}} \sup_{\widetilde{\m{z}}\in \widetilde{S}_r}  \left \langle \widetilde{\m{g}}, \widetilde{\m{z}} \right \rangle
		= \frac{C}{\sqrt{m}} \bE_{\m{g}} \sup_{\m{z}\in S_r} \Re \left \langle \m{g}, \m{z} \right \rangle
		\le 2C \sqrt{\frac{ r \sbra{ c_s + \ln N } \ln N }{m}}, \label{eq:gamma2}
	}
	where the first equality follows from \eqref{eq:gamma2=}, the first inequality follows from Lemma \ref{lem:Talagrand}, and the last one follows from Lemma \ref{lem:width}. 
	
	Putting \eqref{eq:dFd2} and \eqref{eq:gamma2} into \eqref{eq:EVU} yields 
	\equ{
		\begin{split}
			E & \le 4C^2 \frac{r \sbra{ c_s + \ln N } \ln N}{m} +  \frac{2C \sqrt{r \sbra{ c_s + \ln N } \ln N} + 1}{\sqrt{m}}, \\
			V & \le 2C \frac{\sqrt{r \sbra{ c_s + \ln N } \ln N}}{m} + \frac{1}{\sqrt{m}}, \\
			U & \le \frac{1}{m}.
		\end{split} \label{eq:EVU_bound}
	}
	
	To satisfy the constraints in \eqref{eq:EVU_condtion}, we impose them separately on each term of the upper bounds in \eqref{eq:EVU_bound}, which yields sufficient lower bounds on the sample size $m$. Collecting all requirements and taking the dominant one, we arrive at
	\equ{
		m \ge C \delta^{-2} \mbra{ r \sbra{ c_s + \ln N } \ln N  + \ln \sbra{2\varepsilon^{-1}} },
	}
	for some new universal constant $C>0$. In particular, let $\varepsilon = N^{-\ln N}$, for $c_s=1$, this reduces to $m \ge C \delta^{-2} r \ln^2 N$ with probability at least $1-N^{-\ln N}$.
\end{proof}

\subsection{Proof of Theorem \ref{thm:guarantee1}} \label{sec:recovery1}
We first prove that the iterative procedure enjoys local convergence, and then show that, under the one-step hard thresholding initialization, FIHT converges to the true signal. 

In Lemma \ref{lem:local}, we show that, provided the initial estimate lies within the basin of attraction, FIHT converges linearly to the true signal $\m{z}_\star = \cD \m{x}_\star$.

\begin{lemma}[Local convergence] \label{lem:local}
	Assume that $\frobn{\m{L}_0 - \cG\m{z}_\star} \le \frac{1}{4} \sigma_{\text{min}} \sbra{\cG\m{z}_\star}$ and $\m{B}$ satisfies the LRIP of rank $4r$ with constant $\delta_{4r} < \frac{1}{8}$. Then the iterate $\m{z}_t$ in Algorithm \ref{alg:FIHT} satisfies  
	\equ{
		\twon{\m{z}_t - \m{z}_\star} \le v^t \frobn{\m{L}_0 - \cG\m{z}_\star},
	}
	with constant $v < 1$.
\end{lemma}
\begin{proof}
	The proof is deferred to Appendix \ref{subsec:local}.
\end{proof}

In Lemma \ref{lem:init1}, we provide a bound on the error of the initial estimate $\m{L}_0$ from the true Hankel matrix $\cG\m{z}_\star$.
\begin{lemma} \label{lem:init1}
	Let $\m{L}_0 = \cT_r\sbra{ \cG \m{B}^H \m{y} }$ where $\m{y}=\m{B}\m{z}_\star$. Then, 
	\equ{
		\frobn{\m{L}_0 - \cG\m{z}_{\star}} \le 2 \delta_{2r} \frobn{\cG\m{z}_{\star}}. \label{eq:L0gz}
	}
\end{lemma}
\begin{proof}
	The proof is deferred to Appendix \ref{append:init1}.
\end{proof}

\begin{proof}[Proof of Theorem \ref{thm:guarantee1}]
	It follows from Lemma \ref{lem:init1} that
	$\frobn{\m{L}_0 - \cG\m{z}_{\star}} \le 2 \delta_{2r} \frobn{\cG\m{z}_{\star}} \le 2 \delta_{2r} \sqrt{r} \sigma_{\text{max}}\sbra{\cG\m{z}_{\star}}$. To ensure that the condition $\frobn{\m{L}_0 - \cG\m{z}_{\star}} \le \frac{1}{4} \sigma_{\text{min}}\sbra{\cG\m{z}_{\star}}$ in Lemma~\ref{lem:local} is satisfied, it is sufficient to require $\delta_{2r} \le \frac{1}{8\sqrt{r}\kappa}$. Moreover, Lemma~\ref{lem:local} requires $\delta_{4r} < \frac{1}{8}$. Therefore, it suffices to show that $\delta_{4r} \le \frac{1}{64\sqrt{r}\kappa}$, or equivalently,  $\delta \le \frac{1}{64\sqrt{r}\kappa}$ in Theorem \ref{thm:LRIP}. Putting it into \eqref{eq:m} and setting $\varepsilon = N^{-\sbra{c_s + \ln N}}$ completes the proof.
\end{proof}

\subsection{Proof of Theorem \ref{thm:guarantee2}} \label{sec:recovery2}
In Lemma \ref{lem:init2}, we bound the spectral norm error of the initial estimate in Algorithm \ref{alg:Init-resamp} relative to the ground truth.

\begin{lemma} \label{lem:init2}
	Assume $\m{B}$ has i.i.d.\ entries following $\cC\cN\sbra{0,1/m}$. Then, the initial estimate $\widetilde{\m{L}}_0$ in Algorithm \ref{alg:Init-resamp} satisfies
	\equ{
		\twon{\widetilde{\m{L}}_0 - \cG\m{z}_\star} \le \frac{1}{16} \sigma_{\text{min}} \sbra{ \cG\m{z}_\star },
	} 
	with probability at least $1 - 2 N^{-\sbra{c_s + \ln N}}$, provided $\widetilde{m} \ge C \kappa^2 r \max\lbra{ c_s, \ln N } \ln N$ for some universal constant $C>0$.
\end{lemma}

\begin{proof}
	The proof is deferred to Appendix \ref{subsec:init2}.
\end{proof}

Based on Lemma \ref{lem:init2}, Lemma \ref{lem:contraction} establishes a spectral norm upper bound on the estimation error of the output of Algorithm \ref{alg:Init-resamp}.

\begin{lemma} \label{lem:contraction}
	Assume $\m{B}$ has i.i.d.\ entries following $\cC\cN\sbra{0,1/m}$. Then, the output $\widetilde{\m{L}}_T$ of Algorithm \ref{alg:Init-resamp} satisfies
	\equ{
		\twon{\widetilde{\m{L}}_T - \cG\m{z}_\star} 
		\le 2^{-T-4} \sigma_{\text{min}} \sbra{\cG\m{z}_\star}, \label{eq:LTgz}
	}
	with probability at least $1 - \sbra{2+T} N^{-\sbra{c_s + \ln N}}$, provided $\widetilde{m} \ge C \kappa^2 r \max \lbra{c_s, \ln N} \ln N$ for some universal constant $C>0$.
\end{lemma}

\begin{proof}
	The proof is deferred to Appendix \ref{subsec:contraction}.
\end{proof}

\begin{proof}[Proof of Theorem \ref{thm:guarantee2}]
	We have
	\equ{
		\begin{split}
			\frobn{\m{L}_0 - \cG\m{z}_\star} 
			& = \frobn{\widetilde{\m{L}}_T - \cG\m{z}_\star} \\
			& \le \sqrt{2r} \twon{\widetilde{\m{L}}_T - \cG\m{z}_\star} \\
			& \le \sqrt{2r} 2^{ -\lceil \ln r \rceil - 3 } \sigma_{\text{min}} \sbra{\cG\m{z}_\star} \\
			& \le \frac{\sqrt{2}}{8} r^{\frac{1}{2} + \ln \frac{1}{2} }  \sigma_{\text{min}} \sbra{\cG\m{z}_\star} \\
			& \le \frac{1}{4} \sigma_{\text{min}} \sbra{\cG\m{z}_\star},
		\end{split} \label{eq:L0gz2}
	}
	where the second inequality follows from Lemma \ref{lem:contraction} and $T \ge \lceil \ln r \rceil - 1 $.
	For $r\ge2$, $T+1=\lceil\ln r\rceil\le2\ln r$ and $m=(T+1)\widetilde m$, so condition~\eqref{eq:m_guarantee2} ensures the per-batch sample requirements in Lemmas~\ref{lem:init2} and~\ref{lem:contraction}.
	Under the same sample condition, Theorem~\ref{thm:LRIP}, applied at rank $4r$ with $\delta=1/16$ and $\varepsilon=N^{-(c_s+\ln N)}$, gives $\delta_{4r}\le1/16$ with probability at least $1-N^{-(c_s+\ln N)}$.
	On the intersection of this LRIP event and the event in Lemma~\ref{lem:contraction}, \eqref{eq:L0gz2} and Lemma~\ref{lem:local} yield \eqref{eq:zz}.
	A union bound gives failure probability at most $(T+3)N^{-(c_s+\ln N)}\le(3+\ln r)N^{-(c_s+\ln N)}$.
	For $r=1$, take $T=0$. The stated sample condition ensures the same initialization and LRIP events, so the argument above applies with failure probability at most $3N^{-(c_s+\ln N)}$.
\end{proof}

\section{Conclusion} \label{sec:conclusion}
In this paper, we studied spectral compressed sensing from linear measurements by reformulating it as the recovery of a low-rank Hankel matrix. We proposed the LRIP as a general theoretical tool for analyzing Hankel-structured recovery under Gaussian measurements. Based on this framework, we introduced the FIHT algorithm, which achieves exact recovery with near-optimal sample complexity of order $r \ln r$, improving from order $r^2$ in previous nonconvex Hankel methods. Numerical experiments demonstrated that FIHT achieves competitive phase transitions while being significantly more computationally efficient than existing convex optimization methods. 

\section*{Appendix}
\appendix 
\section{Proof of Lemma \ref{lem:norm}} \label{append:norm}
Our proof relies on the following norm bound for random series with matrix coefficients.
\begin{lemma} [{\cite[Theorem 4.1.1]{tropp2015introduction}}] \label{lem:trop}
	Let $\m{A}_1,\ldots,\m{A}_K \in \bC^{n_1 \times n_2}$. Assume that $g_1,\ldots,g_K$ are i.i.d.\ $\cC\cN \sbra{0,1}$.  
	Introduce the matrix Gaussian series $\m{Z} = \sum_{k=1}^K g_k \m{A}_k$.
	Let $\sigma^2\sbra{\m{Z}}$ be the matrix variance statistic of the sum:	
	\equ{
		\sigma^2\sbra{\m{Z}} := \max\lbra{ \twon{\bE \mbra{ \m{Z} \m{Z}^H }}, \twon{\bE \mbra{ \m{Z}^H \m{Z} }} }.
	}
	Then,
	\equ{
		\bE \twon{\m{Z}} \le \sqrt{2 \sigma^2\sbra{\m{Z}} \ln \sbra{n_1 + n_2} }.
	}
\end{lemma}

\begin{proof}[Proof of Lemma \ref{lem:norm}]
	Let $\m{e}_a$ be the $a$-th canonical basis of $\bR^N$ for $a=0,\ldots,N-1$. Denote
	\equ{
		\m{H}_a := \frac{1}{\sqrt{w_a}} \cH \m{e}_a \in \bC^{n_1 \times n_2}.
	}
	Recall that $w_a = \#\lbra{\sbra{i,j} | i+j = a, \ i\in \mbra{0,n_1-1}, \ j\in \mbra{0,n_2-1}, \ i,j\in \bZ}$. It follows from the definition of $\cG$ that $\cG \m{g}$ can be represented as a matrix Gaussian series as follows $$\cG \m{g} = \sum^{N-1}_{a=0} g_a \m{H}_a.$$ To apply Lemma \ref{lem:trop}, we need to bound the matrix variance statistic $\sigma^2\sbra{\cG\m{g}}$. Specifically,
	\equ{
		\begin{split}
			\bE \mbra{ \cG\m{g} \sbra{\cG\m{g}}^H } 
			= \sum^{N-1}_{a=0} \sum^{N-1}_{b=0} \bE \mbra{g_a \overline{g_b}} \m{H}_a \m{H}_b^T
			= \sum^{N-1}_{a=0} \sum^{N-1}_{b=0} \delta_{ab} \m{H}_a \m{H}_b^T
			= \sum^{N-1}_{a=0} \m{H}_a\m{H}_a^T
			= \sum^{N-1}_{a=0} \frac{1}{w_a} \cH \m{e}_a \sbra{\cH \m{e}_a}^T,
		\end{split} \label{eq:Egg}
	}
	where $\delta_{ab}$ denotes the Kronecker delta and the second equality holds since $\lbra{g_a}$ are i.i.d.\ $\cC\cN\sbra{0,1}$.
	Denote $S(a) := \lbra{i\in \bZ_{\ge0} : i \in [0,n_1 -1], a-i \in [0,n_2-1]}$. It follows that
	\equ{
		\cH \m{e}_a = \sum_{ \substack{i+j=a \\ i\in [0,n_1-1], j \in [0,n_2-1] } } \m{e}_i \m{e}_j^T = \sum_{i \in S(a) } \m{e}_i \m{e}_{a-i}^T.
	}
	Then, we arrive at
	\equ{
		\begin{split}
			\bE \mbra{ \cG\m{g} \sbra{\cG\m{g}}^H } 
			& = \sum^{N-1}_{a=0} \frac{1}{w_a} \sum_{i \in S(a) } \m{e}_i \m{e}_{a-i}^T \sum_{j \in S(a) } \m{e}_{a-j} \m{e}_j^T \\
			& = \sum^{N-1}_{a=0} \frac{1}{w_a} \sum_{i \in S(a) } \sum_{j \in S(a) } \m{e}_i \delta_{ij} \m{e}_j^T \\
			& = \sum^{N-1}_{a=0} \frac{1}{w_a} \sum_{i \in S(a) } \m{e}_i \m{e}_i^T \\
			& = \sum^{n_1-1}_{i=0} \sbra{ \sum^{i+n_2-1}_{a=i} \frac{1}{w_a} } \m{e}_i \m{e}_i^T,
		\end{split}
	}
	which is a diagonal matrix. Consequently, we have
	\equ{
		\begin{split}
			\twon{ \bE \mbra{ \cG\m{g} \sbra{\cG\m{g}}^H } } 
			& = \max_{i\in [0,n_1-1]} \sum^{i+n_2-1}_{a=i} \frac{1}{w_a} \\
			& = \sum^{n_2-1}_{a=0} \frac{1}{w_a} \\
			& = 1 + \frac{1}{2} + \ldots + \frac{1}{n_1} + \underbrace{\frac{1}{n_1} + \ldots + \frac{1}{n_1}}_{n_2 - n_1} \\
			& \le \ln \sbra{n_1} + 1 + \frac{n_2-n_1}{n_1} \\
			& = \ln \sbra{n_1} + \frac{n_2}{n_1},
		\end{split} \label{eq:gg1}
	}
	where we use $n_1 \le n_2$ and the fact
	\equ{
		w_i =
		\begin{cases}
			i+1, & \text{if } i \in [0,n_1-1], \\
			n_1, & \text{if } i \in [n_1,n_2-1], \\
			n_1+n_2-1-i, & \text{if } i \in [n_2, n_1+n_2-2].
		\end{cases}
	}
	
	Similarly, we have
	\equ{
		\begin{split}
			\bE \mbra{ \sbra{\cG\m{g}}^H \cG\m{g} } & = \sum^{n_2-1}_{i=0} \sbra{ \sum^{i+n_1-1}_{a=i} \frac{1}{w_a} } \m{e}_i \m{e}_i^T,
		\end{split}
	}
	and thus
	\equ{
		\twon{\bE \mbra{ \sbra{\cG\m{g}}^H \cG\m{g} }} =  \max_{i\in [0,n_2-1]} \sum^{i+n_1-1}_{a=i} \frac{1}{w_a} = \sum^{n_1-1}_{a=0} \frac{1}{w_a} = 1 + \frac{1}{2} + \ldots + \frac{1}{n_1} \le \ln \sbra{n_1} + 1. \label{eq:gg2}
	}
	
	Combining \eqref{eq:gg1} and \eqref{eq:gg2} yields that $\sigma^2\sbra{\cG\m{g}} \le \ln \sbra{n_1} + \frac{n_2}{n_1}$.
	Applying Lemma \ref{lem:trop}, for $N>1$, we have
	\equ{
		\bE \twon{\cG\m{g}} \le \sqrt{2  \sigma^2\sbra{\cG\m{g}} \ln \sbra{n_1+n_2} } \le \sqrt{ 2 \sbra{ \frac{n_2}{n_1} + \ln n_1 } \ln \sbra{N+1} } \le 2 \sqrt{ \sbra{ c_s + \ln N } \ln N }.
	}
\end{proof}

\section{Implementation of FIHT} \label{append:FIHT}
For completeness, we present the implementation details of FIHT, which follow those in \cite{cai2019fast}. 
Assume the rank-$r$ $n_1\times n_2$ matrix $\m{L}_t$ is stored by its reduced SVD $\m{L}_t = \m{U}_t \m{\Sigma}_t \m{V}_t^H$, where $\m{U}_t \in \bC^{n_1 \times r}$ and $\m{V}_t \in \bC^{n_2 \times r}$. The tangent space $T_t$ at $\m{L}_t$ is given by
\equ{
	T_t := \lbra{\m{U}_t \m{E}^H + \m{F} \m{V}_t^H: \m{E}\in \bC^{n_2 \times r}, \m{F} \in \bC^{n_1\times r} }. \label{eq:Tangentspace}
}
Denote the Hankel matrix $\m{H}_t := \cG \sbra{ \m{z}_t - \m{g}_t }$. The matrix $\m{W}_t$ is stored by the following decomposition
\equ{
	\begin{split}
		\m{W}_t = \cP_{T_t} \m{H}_t 
		& = \m{U}_t \m{U}_t^H \m{H}_t + \m{H}_t \m{V}_t \m{V}_t^H - \m{U}_t \m{U}_t^H \m{H}_t \m{V}_t \m{V}_t^H \\
		& = \m{U}_t \underbrace{ \m{U}_t^H \m{H}_t \m{V}_t }_{\m{C} \in \bC^{r\times r} } \m{V}_t^H + \m{U}_t \underbrace{ \m{U}_t^H \m{H}_t \sbra{ \m{I} - \m{V}_t \m{V}_t^H} }_{\m{X}^H \in \bC^{r\times n_2} }  +  \underbrace{ \sbra{ \m{I} - \m{U}_t \m{U}_t^H } \m{H}_t \m{V}_t }_{\m{Y} \in \bC^{n_1\times r}} \m{V}_t^H  \\
		& = \m{U}_t \m{C} \m{V}_t^H + \m{U}_t \m{X}^H + \m{Y} \m{V}_t^H \\
		& \overset{(a)}{=} \m{U}_t \m{C} \m{V}_t^H + \m{U}_t \m{R}_1^H \m{Q}_1^H + \m{Q}_2 \m{R}_2 \m{V}_t^H \\
		& = \begin{bmatrix} \m{U}_t & \m{Q}_2 \end{bmatrix} \begin{bmatrix} \m{C} & \m{R}_1^H \\ \m{R}_2 & \m{0} \end{bmatrix} \begin{bmatrix} \m{V}_t & \m{Q}_1 \end{bmatrix}^H \\
		& \overset{(b)}{=} \begin{bmatrix} \m{U}_t & \m{Q}_2 \end{bmatrix} \m{U}_c \m{\Sigma}_c \m{V}_c^H \begin{bmatrix} \m{V}_t & \m{Q}_1 \end{bmatrix}^H \\
		& = \sbra{ \begin{bmatrix} \m{U}_t & \m{Q}_2 \end{bmatrix} \m{U}_c } \m{\Sigma}_c \sbra{ \begin{bmatrix} \m{V}_t & \m{Q}_1 \end{bmatrix} \m{V}_c }^H
	\end{split} \label{eq:Wt}
}
where in step $(a)$ we let $\m{X} = \m{Q}_1\m{R}_1$ and $\m{Y} = \m{Q}_2 \m{R}_2$ be the QR factorizations of $\m{X}$ and $\m{Y}$, respectively, in step $(b)$ we let the SVD of the $2r\times 2r$ matrix $\begin{bmatrix} \m{C} & \m{R}_1^H \\ \m{R}_2 & \m{0} \end{bmatrix} $ is $ \m{U}_c \m{\Sigma}_c \m{V}_c^H$, and the last line is the SVD of $\m{W}_t$ since $\m{V}_t^H \m{Q}_1 = \m{0}$ and $\m{U}_t^H \m{Q}_2 = \m{0}$. The Hankel matrix-matrix multiplications $\m{H}_t \m{V}_t$ and $\m{U}_t^H\m{H}_t$ in $\m{C}$, $\m{X}$ and $\m{Y}$ can be computed via the FFT without forming $\m{H}_t$ explicitly, requiring $\cO\sbra{ r N \log N }$ flops. The additional matrix multiplications cost $\cO\sbra{r^2 N}$ flops and the SVD of the $2r \times 2r$ matrix requires $\cO\sbra{r^3}$ flops. In addition, the $\cP_{T_t} \cG \m{g}_t$ in the stepsize $\alpha_t$ can be efficiently stored and computed in the same manner as in \eqref{eq:Wt}.

\section{Proof of Lemma \ref{lem:local}} \label{subsec:local}
The following lemmas will be used in the proof.

\begin{lemma} \label{lem:I-PT}
	Let $T_t$ be the tangent space of a rank-$r_1$ matrix manifold. Let $\m{X}$ be another rank-$r_2$ matrix. Then, 
	\equ{
		\frobn{\cP_{T_t} \cG \sbra{ \m{B}^H \m{B} - \m{I} } \cG^*\m{X} } \le \delta_{2r_1+r_2} \frobn{\m{X}}.
	}
\end{lemma}
\begin{proof}
	The proof is deferred to Appendix \ref{append:I-PT}.
\end{proof}

\begin{lemma}[{\cite[Lemma 4.1]{wei2016guarantees}\cite[Lemma 12]{cai2025fast}}] \label{lem:wei}
	Let $\m{L}_t$ be a rank-$r$ matrix and $T_t$ be the tangent space of the rank-$r$ matrix manifold at $\m{L}_t$. Let $\m{X}$ be another rank-$r$ matrix. Then,
	\equ{
		\frobn{ \sbra{\cI - \cP_{T_t}} \m{X} } \le \frac{\frobn{\m{L}_t - \m{X}}^2}{\sigma_{\text{min}}\sbra{ \m{X} }}, \qquad \twon{ \sbra{\cI - \cP_{T_t}}  \m{X} } \le \frac{\twon{\m{L}_t - \m{X} }^2}{\sigma_{\text{min}}\sbra{ \m{X} }}.
	}
\end{lemma}

\begin{proof}[Proof of Lemma \ref{lem:local}]
	Note that 
	\equ{
		\twon{\m{z}_{t+1}-\m{z}_\star} = \twon{ \cG^*\sbra{ \m{L}_{t+1} - \cG\m{z}_\star } } \le \frobn{\m{L}_{t+1} - \cG\m{z}_\star}, \label{eq:zt+1}
	} 
	where the facts $\m{z}_{t+1} = \cG^*\m{L}_{t+1}$, $\cG^*\cG = \cI$, and $\twon{\cG^*}\le 1$ are used.
	We have
	\equ{
		\begin{split}
			\frobn{\m{L}_{t+1} - \cG\m{z}_\star} 
			& \le \frobn{\m{L}_{t+1} - \m{W}_t} + \frobn{\m{W}_t - \cG\m{z}_\star} \\
			& \overset{(a)}{\le} 2 \frobn{\m{W}_t - \cG\m{z}_\star} \\
			& = 2 \frobn{ \cP_{T_t} \cG \sbra{ \m{z}_t - \m{B}^H\sbra{\m{B}\m{z}_t - \m{y}} }  - \cG\m{z}_\star} \\
			& \overset{(b)}{=} 2 \frobn{ \cP_{T_t} \cG \sbra{ \cG^* \m{L}_t -  \m{B}^H \m{B} \cG^* \m{L}_t + \m{B}^H \m{B} \m{z}_\star } - \cG \m{z}_\star } \\
			& = 2 \frobn{ \cP_{T_t} \cG \mbra{ \sbra{ \m{I} -  \m{B}^H \m{B} } \cG^* \m{L}_t + \m{B}^H \m{B} \m{z}_\star } - \cG \m{z}_\star } \\
			& \overset{(c)}{=} 2 \frobn{ \cP_{T_t} \cG \sbra{\m{I} - \m{B}^H\m{B}} \cG^* \sbra{ \m{L}_t - \cG \m{z}_\star } + \sbra{ \cP_{T_t} - \cI } \sbra{\cG\m{z}_\star} } \\
			& \le 2 \frobn{ \cP_{T_t} \cG \sbra{\m{I} - \m{B}^H\m{B}} \cG^* \sbra{ \m{L}_t - \cG \m{z}_\star } } + 2 \frobn{ \sbra{ \cP_{T_t} - \cI } \sbra{\cG\m{z}_\star} } \\
			& \overset{(d)}{\le} \sbra{2 \delta_{4r} + \frac{2 \frobn{\m{L}_t - \cG\m{z}_\star}}
				{\sigma_{\text{min}}\sbra{\cG\m{z}_\star}} } \frobn{\m{L}_t - \cG \m{z}_\star},
		\end{split} \label{eq:I1I2I3}
	}
	where step $(a)$ holds since both $\m{L}_{t+1}$ and $\cG\m{z}_\star$ have rank $r$ and $\m{L}_{t+1}$ is the best rank-$r$ approximation to $\m{W}_t$, step $(b)$ follows from $\m{y}=\m{B}\m{z}_\star$ and $\m{z}_t = \cG^* \m{L}_t$, step $(c)$ follows from $\cG^* \cG = \cI$, and step $(d)$ uses Lemma \ref{lem:I-PT} and Lemma \ref{lem:wei}.
	
	Since $\delta_{4r} \le 1/8$, we have $
		v:=
		2 \delta_{4r}
		+
		\frac12
		<1$.
	We now prove by induction that
	\equ{
		\frobn{\m{L}_t-\cG\m{z}_\star}
		\le
		\frac14
		\sigma_{\text{min}}\sbra{\cG\m{z}_\star},
		\qquad \forall t\ge0.
	}
	The claim holds for $t=0$ by assumption. Suppose it holds for some $t$. Then,
	from \eqref{eq:I1I2I3},
	\equ{
		\begin{split}
			\frobn{\m{L}_{t+1} - \cG\m{z}_\star}
			\le
			v \frobn{\m{L}_{t} - \cG\m{z}_\star}
			\le
			\frac{v}{4}
			\sigma_{\text{min}}\sbra{\cG\m{z}_\star}
			\le
			\frac14
			\sigma_{\text{min}}\sbra{\cG\m{z}_\star},
		\end{split}
	}
	since $v<1$. Hence the claim holds for all $t\ge0$. Consequently,
	\equ{
		\frobn{\m{L}_{t} - \cG\m{z}_\star}
		\le
		v^t
		\frobn{\m{L}_{0} - \cG\m{z}_\star}.
		\label{eq:19}
	}
	Substituting \eqref{eq:19} into \eqref{eq:zt+1} completes the proof.
\end{proof}

\section{Proof of Lemma \ref{lem:init1}}  \label{append:init1}
\begin{lemma} \label{lem:Proj_RIP}
	Let $\cS\subseteq\bC^{n_1\times n_2}$ be a linear subspace such that $
	\rank(\m{Z})\le 2r$ for $\m{Z}\in\cS$, and $\cP_{\cS}$ denote the orthogonal projector onto $\cS$.
	Then,
	\equ{
		\twon{
			\cP_{\cS}\cG
			\sbra{\m{B}^H\m{B}-\m{I}}
			\cG^*\cP_{\cS}
		}
		\le \delta_{2r}.
	}
\end{lemma}
\begin{proof}
	The proof is deferred to Appendix \ref{append:Proj_RIP}.
\end{proof}

The proof of Lemma \ref{lem:init1} is inspired by \cite{wei2016guarantees}. Let $\m{Q}\in \bC^{n_1 \times 2r}$ have orthonormal columns forming a basis for the sum of the column spaces of $\m{L}_0$ and $\cG\m{z}_\star$. Let $\m{Q}^{\perp}$ be the complement of $\m{Q}$. Let $\cP_{\m{Q}}\sbra{\m{X}} := \m{Q}\m{Q}^H \m{X}$ be the orthogonal projection. Then, we have
\equ{
	\begin{split}
		\frobn{ \m{L}_0 - \cG \m{B}^H \m{y} }^2 
		& = \frobn{ \m{L}_0 - \cP_{\m{Q}} \sbra{\cG \m{B}^H \m{y}} }^2 + \frobn{\cP_{\m{Q}^{\perp}} \sbra{\cG \m{B}^H \m{y}} }^2, \\
		\frobn{ \cG\m{z}_\star - \cG \m{B}^H \m{y} }^2 
		& = \frobn{ \cG\m{z}_\star - \cP_{\m{Q}} \sbra{\cG \m{B}^H \m{y}} }^2 + \frobn{\cP_{\m{Q}^{\perp}} \sbra{\cG \m{B}^H \m{y}} }^2.
	\end{split} \label{eq:L0gzstar}
}
Since both $\m{L}_0$ and $\cG\m{z}_\star$ are rank-$r$ and $\m{L}_0$ is the best rank-$r$ approximation of $\cG\m{B}^H\m{y}$, we have 
\equ{
	\frobn{ \m{L}_0 - \cG \m{B}^H \m{y} } \le\frobn{ \cG\m{z}_\star - \cG \m{B}^H \m{y} }. \label{eq:L0gzstar2}
} 
Combining \eqref{eq:L0gzstar} and \eqref{eq:L0gzstar2} yields $\frobn{ \m{L}_0 - \cP_{\m{Q}} \sbra{\cG \m{B}^H \m{y}} } \le \frobn{ \cG\m{z}_\star - \cP_{\m{Q}} \sbra{\cG \m{B}^H \m{y}} } $. Consequently, we have
\equ{
	\begin{split}
		\frobn{\m{L}_0 - \cG\m{z}_{\star}} 
		& \le \frobn{\m{L}_0 - \cP_{\m{Q}} \sbra{\cG \m{B}^H \m{y}} } + \frobn{ \cP_{\m{Q}} \sbra{\cG \m{B}^H \m{y}} - \cG\m{z}_{\star}} \\
		& \le 2 \frobn{ \cP_{\m{Q}} \sbra{\cG \m{B}^H \m{y}} - \cG\m{z}_{\star}} \\
		& = 2 \frobn{ \cP_{\m{Q}} \cG \sbra{ \m{B}^H \m{B} - \m{I} } \cG^* \cP_{\m{Q}} \sbra{ \cG \m{z}_{\star} } } \\
		& \le 2 \delta_{2r} \frobn{ \cG\m{z}_{\star} },
	\end{split}
}
where the equality follows from $\m{y}=\m{B}\m{z}_\star$ and $\cP_{\m{Q}}\sbra{\cG \m{z}_\star} = \cG \m{z}_\star$, and the last inequality follows from Lemma \ref{lem:Proj_RIP}.

\section{Proof of Lemma \ref{lem:init2}} \label{subsec:init2}
\begin{lemma}[{\cite[Theorem 2.8.1]{vershynin2018high}, Bernstein's inequality}]
	\label{lem:bernstein}
	Let $X_1,\ldots,X_m$ be independent, mean-zero,
	sub-exponential random variables whose sub-exponential
	norms are bounded by a common absolute constant.
	Then, for every $t\ge0$,
	\equ{
		\bP\left\{
		\abs{\frac1m\sum_{i=1}^m X_i}\ge t
		\right\}
		\le2\exp \sbra{ -\frac mC\min\{t^2,t\} },
	}
	where $C>0$ is an absolute constant depending only
	on the common bound on the sub-exponential norms.
\end{lemma}

By the rotational invariance of Gaussian random variables, we obtain Lemma~\ref{lem:indep}.

\begin{lemma} \label{lem:indep}
	Let $\m{g} \in \bC^N$ be a random vector whose entries are i.i.d.\ $\cC\cN\sbra{0,\sigma^2}$ and let $\m{w} \in \mathbb{C}^N$ be a nonzero vector. Define the orthogonal projection operators 
	\equ{
		\cP_{\m{w}} \sbra{\m{g}} := \frac{\langle \m{w}, \m{g} \rangle}{ \twon{\m{w}}^2 } \m{w}, 
		\qquad \cP_{\m{w}}^\perp \sbra{\m{g}} := \m{g} - \cP_{\m{w}} \sbra{\m{g}}. \label{eq:Pw}
	}
	Then, $\cP_{\m{w}} \sbra{\m{g}}$ and $\cP_{\m{w}}^\perp \sbra{\m{g}}$ are independent.
\end{lemma}
\begin{proof}
	The proof is deferred to Appendix \ref{append:indep}.
\end{proof}

Based on Lemmas \ref{lem:bernstein} and \ref{lem:indep}, we obtain Lemma \ref{lem:KeySpectralNormBound}, which provides a non-uniform spectral norm bound for any fixed rank-$r$ matrix independent of the measurement matrix, which will play a key role in the proof of Lemmas \ref{lem:init2} and \ref{lem:contraction}.

\begin{lemma} \label{lem:KeySpectralNormBound}
	Let $0<c<1$. Let $\m{B} \in \bC^{m\times N}$ have i.i.d.\ entries following $\cC\cN\sbra{0,1/m}$. Let $\m{X}$ be a rank-$r$ matrix that is statistically independent of $\m{B}$. Assume that $m \ge C c^{-2} r \max\lbra{c_s, \ln N} \ln N$ for some absolute constant $C>0$. Then, with probability at least $1 - N^{- \sbra{c_s + \ln N} }$, it holds that
	\equ{
		\twon{ \cG \sbra{ \m{B}^H \m{B} - \m{I} } \cG^* \m{X} } \le c \twon{\m{X}},
	}
	where $C>0$ is a universal constant.
\end{lemma}
\begin{proof}
	The proof is deferred to Appendix \ref{append:KeySpectralNormBound}.
\end{proof}

	\begin{proof}[Proof of Lemma \ref{lem:init2}]
		Since $\widetilde{\m{L}}_0$ is a best rank-at-most-$r$ approximation to $\cG\m{B}_0^H\m{y}_0$ and $\m{y}_0=\m{B}_0\m{z}_\star$, we have
		\equ{
			\twon{\widetilde{\m{L}}_0-\cG\m{z}_\star}
			\le 2\twon{\cG\m{B}_0^H\m{y}_0-\cG\m{z}_\star}
			= 2\twon{\cG(\m{B}_0^H\m{B}_0-\m{I})\cG^*(\cG\m{z}_\star)}.
			\label{eq:initial_spectral_truncation}
		}
		Let $\m{X}=\cG\m{z}_\star$, $m=\widetilde m$, and $c = 1/(32\kappa)$ in Lemma \ref{lem:KeySpectralNormBound}. Then, 
		\equ{
				\twon{\widetilde{\m{L}}_0-\cG\m{z}_\star}
				\le 2c \kappa \sigma_{\text{min}}(\cG\m{z}_\star)
				\le\frac{1}{16}\sigma_{\text{min}}(\cG\m{z}_\star),
		}
		completing the proof.
	\end{proof}

\section{Proof of Lemma \ref{lem:contraction}} \label{subsec:contraction}
	As in \eqref{eq:I1I2I3}, using the fact $\m{y}_{t+1} = \m{B}_{t+1} \m{z}_\star$, $\widetilde{\m{z}}_t = \cG^* \widetilde{\m{L}}_t$ and $\cG^*\cG = \cI$, we have
	\equ{
		\begin{split}
			\twon{\widetilde{\m{L}}_{t+1} - \cG\m{z}_\star}
			& \le 2 \twon{ \widetilde{\m{W}}_t - \cG\m{z}_\star } \\
			& = 2 \twon{ \cP_{\widetilde{T}_t} \cG \mbra{ \widetilde{\m{z}}_t - \m{B}^H_{t+1} \sbra{\m{B}_{t+1} \widetilde{\m{z}}_t - \m{y}_{t+1}} } - \cG\m{z}_\star } \\
			& = 2 \twon{ \cP_{\widetilde{T}_t} \cG \sbra{ \cG^* \widetilde{\m{L}}_t -  \m{B}_{t+1}^H \m{B}_{t+1} \cG^* \widetilde{\m{L}}_t +  \m{B}_{t+1}^H \m{B}_{t+1} \m{z}_\star } - \cG \m{z}_\star } \\
			& \le 2 \twon{ \cP_{\widetilde{T}_t} \cG \sbra{ \m{I} -  \m{B}_{t+1}^H \m{B}_{t+1} } \cG^* \sbra{ \widetilde{\m{L}}_t - \cG\m{z}_\star } } + 2 \twon{ \sbra{ \cP_{\widetilde{T}_t} - \cI } \sbra{\cG\m{z}_\star} } \\
			& := J_1 + J_2
		\end{split} \label{eq:J1J2}
	}
	
	\textbf{Bounding $J_1$.} We have
	\equ{
		\begin{split}
			J_1 
			& \le 4
			\twon{ \cG \sbra{ \m{I} - \m{B}_{t+1}^H \m{B}_{t+1} } \cG^* \sbra{\widetilde{\m{L}}_t - \cG\m{z}_\star} } \\
			& \le \frac{1}{8} \twon{ \widetilde{\m{L}}_t - \cG\m{z}_\star }.
		\end{split} \label{eq:J1}
	}
	Here, the first inequality follows from $\twon{\cP_{\widetilde{T}_t} \m{X}} = \twon{ \m{\widetilde{U}}_t \m{\widetilde{U}}_t^H \m{X} \sbra{\m{I} - \m{\widetilde{V}}_t \m{\widetilde{V}}_t^H } + \m{X} \m{\widetilde{V}}_t \m{\widetilde{V}}_t^H } \le 2 \twon{\m{X}}$. The second one follows from 
	Lemma~\ref{lem:KeySpectralNormBound} and letting $c\le 1/32$. The conditions required by Lemma \ref{lem:KeySpectralNormBound} are satisfied since the entries of $\m{B}_{t+1}$ are i.i.d.\ $\mathcal{CN}(0,1/\widetilde{m})$ with $\widetilde{m} \ge C \kappa^2 r \max \lbra{c_s, \ln N} \ln N$, $\widetilde{\m{L}}_t - \cG\m{z}_\star$ is independent of $\m{B}_{t+1}$ and
	$\rank\sbra{ \widetilde{\m{L}}_t - \cG\m{z}_\star } \le 2r$. 
	
	\textbf{Bounding $J_2$.} Note that $\widetilde{T}_t$ is the tangent space of the rank-$r$ matrix manifold at $\widetilde{\m{L}}_t$. It follows from Lemma \ref{lem:wei} that
	\equ{
		J_2 \le \frac{2}{\sigma_{\text{min}} \sbra{\cG\m{z}_\star}} \twon{ \widetilde{\m{L}}_t - \cG\m{z}_\star }^2. \label{eq:J2}
	}
	
	Substituting \eqref{eq:J1} and \eqref{eq:J2} into \eqref{eq:J1J2}, we conclude that
	\equ{ \label{eq:Lgz}
		\twon{ \widetilde{\m{L}}_{t+1} - \cG\m{z}_\star }
		\le \sbra{ \frac{1}{8} + \frac{2}{\sigma_{\text{min}} \sbra{\cG\m{z}_\star}} \twon{ \widetilde{\m{L}}_t - \cG\m{z}_\star } } \twon{ \widetilde{\m{L}}_t - \cG\m{z}_\star }.
	}
	Combining Lemma \ref{lem:init2} and \eqref{eq:Lgz} yields \eqref{eq:LTgz}.

\section{Proofs of Technical Lemmas} \label{append:lemmas}

\subsection{Proof of Lemma \ref{lem:I-PT}} \label{append:I-PT}
We have 
\equ{ 
	\begin{split} 
		\frobn{ \cP_{T_t} \cG \sbra{ \m{B}^H \m{B} - \m{I}} \cG^* \m{X} }  
		& = \sup_{\frobn{\m{Z}}=1} \abs{ \left \langle \cP_{T_t} \cG \sbra{  \m{B}^H \m{B} - \m{I} } \cG^* \m{X}, \m{Z} \right \rangle } \\ 
		& = \sup_{\frobn{\m{Z}}=1} \abs{ 
			\left \langle \m{B}\cG^* \m{X}, \m{B}\cG^*\cP_{T_t} \m{Z} \right \rangle  
			- \left \langle \cG^*\m{X}, \cG^* \cP_{T_t} \m{Z} \right \rangle }. 
	\end{split} \label{eq:29} 
} 

Let $\alpha := \frobn{\m{X}}$ and $\beta := \frobn{\cP_{T_t}\m{Z}}$. 
The desired bound is trivial if $\alpha\beta=0$. 
Note that for any $q\in\bC$, there exists $\omega\in\bC$ with $\abs{\omega}=1$ such that $\abs{q}=\omega q$. Choose $\omega\in\bC$ with 
$\abs{\omega}=1$ such that
\equ{
	\begin{split}
		& \abs{
			\left \langle \m{B}\cG^* \m{X}, \m{B}\cG^*\cP_{T_t}\m{Z} \right \rangle
			-
			\left \langle \cG^*\m{X}, \cG^*\cP_{T_t}\m{Z} \right \rangle
		} \\
		& =
		\alpha\beta\, \Re \left(
		\left\langle
		\m{B}\cG^*(\alpha^{-1}\m{X}),
		\m{B}\cG^*(\omega\beta^{-1}\cP_{T_t}\m{Z})
		\right\rangle
		-
		\left\langle
		\cG^*(\alpha^{-1}\m{X}),
		\cG^*(\omega\beta^{-1}\cP_{T_t}\m{Z})
		\right\rangle
		\right).
	\end{split}
}
By the polarization identity,
\equ{
	\begin{split}
		& \abs{
			\left \langle \m{B}\cG^* \m{X}, \m{B}\cG^*\cP_{T_t}\m{Z} \right \rangle
			-
			\left \langle \cG^*\m{X}, \cG^*\cP_{T_t}\m{Z} \right \rangle
		} \\
		& =
		\frac{\alpha\beta}{4}
		\left[
		\twon{
			\m{B}\cG^*
			\sbra{
				\alpha^{-1}\m{X}
				+
				\omega\beta^{-1}\cP_{T_t}\m{Z}
			}
		}^2
		-
		\twon{
			\cG^*
			\sbra{
				\alpha^{-1}\m{X}
				+
				\omega\beta^{-1}\cP_{T_t}\m{Z}
			}
		}^2
		\right. \\
		& \qquad \left.
		-
		\twon{
			\m{B}\cG^*
			\sbra{
				\alpha^{-1}\m{X}
				-
				\omega\beta^{-1}\cP_{T_t}\m{Z}
			}
		}^2
		+
		\twon{
			\cG^*
			\sbra{
				\alpha^{-1}\m{X}
				-
				\omega\beta^{-1}\cP_{T_t}\m{Z}
			}
		}^2
		\right].
	\end{split}
}
Since $\rank(\m{X})\le r_2$ and 
$\rank(\cP_{T_t}\m{Z})\le 2r_1$, we have $
	\rank\sbra{
		\alpha^{-1}\m{X}
		\pm
		\omega\beta^{-1}\cP_{T_t}\m{Z}
	}
	\le 2r_1+r_2$.
Therefore, the LRIP gives
\equ{
	\begin{split}
		\abs{
			\left \langle \m{B}\cG^* \m{X}, \m{B}\cG^*\cP_{T_t}\m{Z} \right \rangle
			-
			\left \langle \cG^*\m{X}, \cG^*\cP_{T_t}\m{Z} \right \rangle
		} 
		& \le
		\frac{\alpha\beta\delta_{2r_1+r_2}}{4}
		\left(
		\frobn{
			\alpha^{-1}\m{X}
			+
			\omega\beta^{-1}\cP_{T_t}\m{Z}
		}^2
		+
		\frobn{
			\alpha^{-1}\m{X}
			-
			\omega\beta^{-1}\cP_{T_t}\m{Z}
		}^2
		\right) \\
		& =
		\frac{\alpha\beta\delta_{2r_1+r_2}}{2}
		\left(
		\frobn{\alpha^{-1}\m{X}}^2
		+
		\frobn{\beta^{-1}\cP_{T_t}\m{Z}}^2
		\right) \\
		& =
		\alpha\beta\delta_{2r_1+r_2}.
	\end{split}
	\label{eq:BB}
}

Putting \eqref{eq:BB} into \eqref{eq:29} yields 
\equ{ 
		\frobn{ \cP_{T_t} \cG \sbra{ \m{B}^H \m{B} - \m{I}} \cG^* \m{X} }  
	\le \sup_{\frobn{\m{Z}}=1} 
		\delta_{2r_1+r_2} \frobn{\m{X}} \frobn{\cP_{T_t}\m{Z}}
		\le \delta_{2r_1+r_2} \frobn{\m{X}}.
}

\subsection{Proof of Lemma \ref{lem:Proj_RIP}} \label{append:Proj_RIP}
	Since $\cP_{\cS} \cG \sbra{ \m{B}^H \m{B} - \m{I} } \cG^* \cP_{\cS}$ is self-adjoint, we have 
	\equ{ 
		\begin{split} 
			\twon{ \cP_{\cS} \cG \sbra{ \m{B}^H \m{B} - \m{I} } \cG^* \cP_{\cS} }  
			& = \sup_{\frobn{\m{Z}}=1} \abs{ \left \langle \sbra{ \cP_{\cS} \cG \sbra{ \m{B}^H \m{B} - \m{I} } \cG^* \cP_{\cS} } \m{Z}, \m{Z} \right \rangle } \\ 
			& = \sup_{\frobn{\m{Z}}=1} \abs{ \twon{\m{B}\cG^*\cP_{\cS}\m{Z}}^2 - \twon{\cG^*\cP_{\cS}\m{Z}}^2 } \\ 
			& \le \sup_{\frobn{\m{Z}}=1} \delta_{2r} \frobn{\cP_{\cS}\m{Z}}^2 \\ 
			& \le \delta_{2r}, 
		\end{split} 
	} 
	where the first inequality follows from the LRIP and $\rank\sbra{\cP_{\cS}\m{Z}} \le 2r$.

\subsection{Proof of Lemma \ref{lem:indep}} \label{append:indep}
	Since $\m{w} \neq \m{0}$, there exists a unitary matrix $\m{U}$ such that $\m{U}\m{w} = \twon{\m{w}} \m{e}_1$, where 
	$\m{e}_1 = (1,0,\dots,0)^T$.  
	Because $\m{g} \sim \cC\cN(\m{0}, \sigma^2\m{I}_N)$ is rotationally invariant, we have $\m{U} \m{g}$ has the same distribution as $\m{g}$.
	Therefore, it suffices to prove the claim for $\m{w} = \m{e}_1$.
	
	Let $\widetilde{\m{g}} := \m{U} \m{g} = (\widetilde{g}_1,\dots,\widetilde{g}_N)^T$, where $\widetilde{g}_1,\dots,\widetilde{g}_N$ are i.i.d.\  $\mathcal{CN}(0,\sigma^2)$.
	For $\m{w} = \m{e}_1$, the projections admit the explicit forms
	$\cP_{\m{e}_1}(\widetilde{\m{g}}) = \widetilde{g}_1 \m{e}_1, 
	\cP_{\m{e}_1}^\perp(\widetilde{\m{g}}) = (0,\widetilde{g}_2,\dots,\widetilde{g}_N)^T$.
	Hence, $(\cP_{\m{e}_1}(\widetilde{\m{g}}), \cP_{\m{e}_1}^\perp(\widetilde{\m{g}}))$ is a jointly complex Gaussian vector, being a
	linear transformation of $\widetilde{\m{g}}$.
	Moreover, the cross-covariance satisfies 
	$\mathbb{E} \mbra{ \cP_{\m{e}_1}(\widetilde{\m{g}})\, \cP_{\m{e}_1}^\perp(\widetilde{\m{g}})^H } = \m{0}$, since $\lbra{\widetilde{g}_n}$ are independent and mean-zero.
	By the independence criterion for jointly complex Gaussian random vectors, zero cross-covariance implies independence. Thus, $\cP_{\m{e}_1}(\widetilde{\m{g}}) $ is independent of $ \cP_{\m{e}_1}^\perp(\widetilde{\m{g}})$.
	Finally, independence is preserved under unitary transformations, and the desired result for general $\m{w}$ follows.

\subsection{Proof of Lemma \ref{lem:KeySpectralNormBound}} \label{append:KeySpectralNormBound}
	We condition on $\m{X}$. If $\cG^*\m{X}=\m0$, the conclusion is immediate, so assume $\cG^*\m{X}\ne\m0$.
	We separates the target into its component along $\cG^*\m{X}$ and its orthogonal component.
	The first is controlled by scalar concentration, while the second becomes Gaussian after conditioning on the scalar coefficients.
	
	The projection in Lemma~\ref{lem:indep}, applied along $\cG^*\m{X}$, gives
	\equ{
		\m{b}_i
		=\frac{\inp{\cG^*\m{X},\m{b}_i}}{\twon{\cG^*\m{X}}^2}\cG^*\m{X}
		+\cP_{\cG^*\m{X}}^\perp(\m{b}_i).
	}
	Then, we have
	\equ{
		\begin{split}
			\m{B}^H\m{B}\cG^*\m{X}
			= \sum_{i=1}^m \m{b}_i\inp{\m{b}_i,\cG^*\m{X}} 
			& = \sum_{i=1}^m \frac{|\inp{\m{b}_i,\cG^*\m{X}}|^2}{\twon{\cG^*\m{X}}^2}\cG^*\m{X} + \sum_{i=1}^m \inp{\m{b}_i,\cG^*\m{X}}\cP_{\cG^*\m{X}}^\perp(\m{b}_i) \\
			& =  \frac{\twon{\m{B}\cG^*\m{X}}^2}{\twon{\cG^*\m{X}}^2} \cG^*\m{X} +\sum_{i=1}^m\inp{\m{b}_i,\cG^*\m{X}}\cP_{\cG^*\m{X}}^\perp(\m{b}_i).
		\end{split}
	}
	It follows that
	\equ{
		\begin{split}
			\twon{\cG(\m{B}^H\m{B}-\m{I})\cG^*\m{X}}
			& \le\abs{\frac{\twon{\m{B}\cG^*\m{X}}^2}{\twon{\cG^*\m{X}}^2}-1}\twon{\cG\cG^*\m{X}}  +\twon{\cG \sum_{i=1}^m\inp{\m{b}_i,\cG^*\m{X}}\cP_{\cG^*\m{X}}^\perp(\m{b}_i) }\\
			& :=K_1 + K_2.
		\end{split}
		\label{eq:J11}
	}

	\textbf{Bounding $K_1$.}
	We have
	\equ{
		\frac{\twon{\m{B}\cG^*\m{X}}^2}{\twon{\cG^*\m{X}}^2}-1
		=\frac1m\sum_{i=1}^m
		\left(
		\frac{m|\inp{\m{b}_i,\cG^*\m{X}}|^2}{\twon{\cG^*\m{X}}^2}-1
		\right).
		\label{eq:input_energy_centering}
	}
	Since $\cG^*\m{X}$ is fixed under the conditioning and the vectors $\m{b}_i$ are independent, we have
	$
		\inp{\m{b}_i,\cG^*\m{X}}
		\sim\ccn\left(0,\frac{\twon{\cG^*\m{X}}^2}{m}\right),
	$ and $
		\frac{\sqrt m\,\inp{\m{b}_i,\cG^*\m{X}}}{\twon{\cG^*\m{X}}}
		\sim\ccn(0,1).
	$
	The squared modulus of a standard complex Gaussian variable is exponential with mean one.
	Thus the variables $m|\inp{\m{b}_i,\cG^*\m{X}}|^2/\twon{\cG^*\m{X}}^2$ are independent exponentials of mean one, and the centered summands in \eqref{eq:input_energy_centering} have uniformly bounded sub-exponential norms.
	
	Applying Lemma \ref{lem:bernstein} gives, for every $t\ge0$,
	\equ{
		\bP\left\{
		\abs{\frac{\twon{\m{B}\cG^*\m{X}}^2}{\twon{\cG^*\m{X}}^2}-1}>t
		\right\}
		\le2\exp \sbra{ -\frac mC\min\{t^2,t\} },
		\label{eq:input_bernstein_tail}
	}
	where $C\ge1$ is a universal constant.
	Assume $m\ge4C(c_s+\ln N)\ln N$, and choose $t=2\sqrt{ C(c_s+\ln N)\ln N / m}$.
	Then $t\le1$, so $\min\{t^2,t\}=t^2$ and $mt^2/C=4(c_s+\ln N)\ln N$.
	Substituting into \eqref{eq:input_bernstein_tail} yields
	\equ{
			\bP\left\{
			\abs{\frac{\twon{\m{B}\cG^*\m{X}}^2}{\twon{\cG^*\m{X}}^2}-1}
			> 2 \sqrt{\frac{C(c_s+\ln N)\ln N}{m}}
			\right\} 
			\le 2\exp \sbra{ -4(c_s+\ln N)\ln N}
			\le \frac12N^{-(c_s+\ln N)}.
	}
	The last inequality uses $(c_s+\ln N)\ln N\ge1$, since $N\ge2$ and $c_s\ge1$.
	Consequently, with probability at least $1-\frac12N^{-(c_s+\ln N)}$,
	\equ{
		\abs{\frac{\twon{\m{B}\cG^*\m{X}}^2}{\twon{\cG^*\m{X}}^2}-1}
		\le2\sqrt{\frac{C(c_s+\ln N)\ln N}{m}}
		\le1.
		\label{eq:input_energy}
	}
	Moreover, we have $
		\twon{\cG\cG^*\m{X}}
		\le\frobn{\cG\cG^*\m{X}}
		=\twon{\cG^*\m{X}}$.
	Combining this inequality with \eqref{eq:input_energy} in the definition of $K_1$ yields, on the same event,
	\equ{
		K_1
		\le 2
		\sqrt{\frac{C(c_s+\ln N)\ln N}{m}}
		\twon{\cG^*\m{X}}.
		\label{eq:K1}
	}
	
	\textbf{Bounding $K_2$.}
	By Lemma~\ref{lem:indep} and independence across $i$,
	$\{\inp{\m{b}_i,\cG^*\m{X}}\}_{i=1}^m$
	is independent of
	$\{\cP_{\cG^*\m{X}}^\perp(\m{b}_i)\}_{i=1}^m$.
	Since $\bE[\m{b}_i\m{b}_i^H]=\m{I}_N/m$, we have
	$
			\bE\left[
			\cP_{\cG^*\m{X}}^\perp(\m{b}_i)
			\bigl(\cP_{\cG^*\m{X}}^\perp(\m{b}_i)\bigr)^H
			\right]
			=\frac1m
			\cP_{\cG^*\m{X}}^\perp
			\bigl(\cP_{\cG^*\m{X}}^\perp\bigr)^H
			=\frac1m\cP_{\cG^*\m{X}}^\perp.
		\label{eq:orthogonal_covariance}
	$
	
	Conditioning on the coefficients
	$\{\inp{\m{b}_i,\cG^*\m{X}}\}_{i=1}^m$
	leaves the projected vectors independent and centered complex Gaussian.
	Their weighted sum is therefore conditionally centered complex Gaussian, with covariance
	\equ{
		\frac1m\sum_{i=1}^m
		|\inp{\m{b}_i,\cG^*\m{X}}|^2
		\cP_{\cG^*\m{X}}^\perp
		=
		\frac{\twon{\m{B}\cG^*\m{X}}^2}{m}
		\cP_{\cG^*\m{X}}^\perp.
		\label{eq:conditional_covariance}
	}
	
	Let $\m{g}\sim\ccn(\m0,\m{I}_N)$ be independent of $\m{X}$ and the coefficients.
	Under the conditioning, $\twon{\m{B}\cG^*\m{X}}$ is fixed, and the vector
	$
	\frac{\twon{\m{B}\cG^*\m{X}}}{\sqrt m}
	\cP_{\cG^*\m{X}}^\perp(\m{g})
	$
	is centered complex Gaussian with the covariance in \eqref{eq:conditional_covariance}.
	Then we have
	\equ{
		\sum_{i=1}^m
		\inp{\m{b}_i,\cG^*\m{X}}
		\cP_{\cG^*\m{X}}^\perp(\m{b}_i)
		\overset d=
		\frac{\twon{\m{B}\cG^*\m{X}}}{\sqrt m}
		\cP_{\cG^*\m{X}}^\perp(\m{g}),
	}
	where $\overset d=$ denotes equality in distribution conditional on $\m{X}$ and the coefficients.
	Applying $\cG$ and taking the spectral norm therefore gives, under the same conditioning,
	\equ{
		K_2
		\overset d=
		\frac{\twon{\m{B}\cG^*\m{X}}}{\sqrt m}
		\twon{\cG\cP_{\cG^*\m{X}}^\perp(\m{g})}.
		\label{eq:orthogonal_conditional_law}
	}
	
	To bound the norm in \eqref{eq:orthogonal_conditional_law}, we first control its expectation and then its deviation above the expectation.
	By Lemma \ref{lem:indep}, the two components in $
		\m{g}
		=\cP_{\cG^*\m{X}}^\perp(\m{g})
		+\cP_{\cG^*\m{X}}(\m{g})$
	are independent, and the parallel component has mean zero.
	Consequently,
	$
			\bE\left[
			\m{g}\mid\cP_{\cG^*\m{X}}^\perp(\m{g})
			\right]
			 =\cP_{\cG^*\m{X}}^\perp(\m{g})
			+\bE\left[\cP_{\cG^*\m{X}}(\m{g})\right]
			=\cP_{\cG^*\m{X}}^\perp(\m{g}).
	$
	By linearity of $\cG$ and conditional Jensen's inequality for the convex spectral norm, we have
	\equ{
			\twon{\cG\cP_{\cG^*\m{X}}^\perp(\m{g})}
			=
			\twon{
				\bE\left[
				\cG\m{g}\mid\cP_{\cG^*\m{X}}^\perp(\m{g})
				\right]
			}
			\le
			\bE\left[
			\twon{\cG\m{g}}
			\mid\cP_{\cG^*\m{X}}^\perp(\m{g})
			\right].
	}
	Taking expectations and applying Lemma~\ref{lem:norm} yields
	\equ{
		\bE\twon{\cG\cP_{\cG^*\m{X}}^\perp(\m{g})}
		\le
		\bE\left[
		\bE\left[
		\twon{\cG\m{g}}
		\mid\cP_{\cG^*\m{X}}^\perp(\m{g})
		\right]
		\right]
		= \bE\twon{\cG\m{g}}
		\le2\sqrt{(c_s+\ln N)\ln N}.
		\label{eq:orthogonal_gaussian_mean}
	}
	
	To turn this expectation bound into a high-probability bound, we verify the Lipschitz property required for Gaussian concentration.
	For any two inputs $\m{g},\m{w}\in\bC^N$,
	\equ{
		\abs{
				\twon{\cG\cP_{\cG^*\m{X}}^\perp(\m{g})}
				-\twon{\cG\cP_{\cG^*\m{X}}^\perp(\m{w})}
			}
			\le
			\twon{\cG\cP_{\cG^*\m{X}}^\perp(\m{g}-\m{w})}
			\le
			\frobn{\cG\cP_{\cG^*\m{X}}^\perp(\m{g}-\m{w})}
			=
			\twon{\cP_{\cG^*\m{X}}^\perp(\m{g}-\m{w})}
			\le\twon{\m{g}-\m{w}}.
	}
	Thus the map
	$\m{g}\mapsto\twon{\cG\cP_{\cG^*\m{X}}^\perp(\m{g})}$
	is $1$-Lipschitz.
	Gaussian concentration \cite[Theorem~5.2.2]{vershynin2018high}, applied to the real and imaginary coordinates, gives, for every $t\ge0$,
	\equ{
			\bP\left\{
			\twon{\cG\cP_{\cG^*\m{X}}^\perp(\m{g})}
			>
			\bE\twon{\cG\cP_{\cG^*\m{X}}^\perp(\m{g})}+t
			\right\} \le2\exp(-t^2/C).
	}

Choose $t=2\sqrt{C(c_s+\ln N)\ln N}$.
Combing with \eqref{eq:orthogonal_gaussian_mean}, we have
\equ{
	\bP\left\{
	\twon{\cG\cP_{\cG^*\m{X}}^\perp(\m{g})}
	>(2+2\sqrt C)\sqrt{(c_s+\ln N)\ln N}
	\right\}
	\le2\exp[-4(c_s+\ln N)\ln N]
	\le\frac12N^{-(c_s+\ln N)}.
	\label{eq:orthogonal_gaussian_tail}
}
For every fixed realization of the coefficients, \eqref{eq:orthogonal_conditional_law} and \eqref{eq:orthogonal_gaussian_tail} imply
\equ{
	K_2
	\le(2+2\sqrt C)
	\frac{\twon{\m{B}\cG^*\m{X}}}{\sqrt m}
	\sqrt{(c_s+\ln N)\ln N}
	\label{eq:conditional_K2_bound}
}
with conditional failure probability at most $\frac12N^{-(c_s+\ln N)}$.
Averaging over the coefficients preserves this failure bound.
On the energy event \eqref{eq:input_energy}, we also have
$\twon{\m{B}\cG^*\m{X}}\le\sqrt2\twon{\cG^*\m{X}}$.
Substituting this into \eqref{eq:conditional_K2_bound} gives
\equ{
	K_2
	\le\sqrt2(2+2\sqrt C)
	\sqrt{\frac{(c_s+\ln N)\ln N}{m}}
	\twon{\cG^*\m{X}}.
	\label{eq:K2}
}
Since the energy event fails with probability at most $\frac12N^{-(c_s+\ln N)}$ and also guarantees \eqref{eq:K1}, a union bound shows that \eqref{eq:K1} and \eqref{eq:K2} hold jointly with probability at least $1-N^{-(c_s+\ln N)}$.

	\textbf{Combining the bounds.}
	Substituting \eqref{eq:K1} and \eqref{eq:K2} into \eqref{eq:J11} gives
	\equ{
		\begin{split}
			\twon{\cG(\m{B}^H\m{B}-\m{I})\cG^*\m{X}}
			& \le\bigl[2\sqrt C+\sqrt2(2+2\sqrt C)\bigr]\sqrt{\frac{(c_s+\ln N)\ln N}{m}}\twon{\cG^*\m{X}}\\
			& \le8\sqrt{\frac{Cr(c_s+\ln N)\ln N}{m}}\twon{\m{X}},
		\end{split}
		\label{eq:lemma18_rate}
	}
	where we used $C\ge1$ and $\twon{\cG^*\m{X}}\le\frobn{\m{X}}\le\sqrt r\twon{\m{X}}$.
	For $0<c<1$, taking $m\ge64Cc^{-2}r(c_s+\ln N)\ln N$ gives the desired bound and satisfies the earlier sample requirement.
	Since $c_s+\ln N\le2\max\{c_s,\ln N\}$, the lemma's sample condition suffices after increasing its constant.
	Averaging over $\m{X}$ preserves the probability bound and completes the proof.

\bibliographystyle{siamplain}
\bibliography{LRIP&FIHT}

\end{document}